%% file: main.tex
\documentclass{amsart}

\usepackage[utf8]{inputenc}
\usepackage{amsmath}
\usepackage{amssymb}
\usepackage{amsthm}
\usepackage{mathrsfs}
\usepackage{enumitem}
\usepackage[marginparwidth=0.8in, margin=1.5in]{geometry}
\usepackage{subfiles}
\usepackage{hyperref}
\usepackage{graphicx}
\usepackage{marginnote}
\usepackage{thmtools}
\usepackage[dvipsnames, hyperref]{xcolor}
\usepackage{float}
\usepackage[mathscr]{eucal}
\usepackage{cleveref}

\hypersetup{
    colorlinks=true,
    linkcolor=Blue,
    citecolor=Blue
  }

\usepackage{todonotes}

\def\Z{\mathbb{Z}}

\def\R{\mathbf{R}}
\def\C{\mathbf{C}}

\def\H{\mathbb{H}}

\def\Isom{\textup{Isom}}

\def\PSL{\textup{PSL}}

\def\del{\partial}
\def\la{\langle}
\def\ra{\rangle}

\def\Int{\textup{Int}}
\def\al{\alpha}

\def\|{\; | \;}
\def\id{\text{id}}

\def\es{\varnothing}

\def\Homeo{\textup{Homeo}}

\newcommand{\join}{\mathrm{join}}

\newtheorem{theorem}{Theorem}[section]
\newtheorem{corollary}[theorem]{Corollary}
\newtheorem{lemma}[theorem]{Lemma}
\newtheorem{proposition}[theorem]{Proposition}

\newtheorem*{claim*}{Claim}
\newtheorem*{theorem A}{Theorem A}
\newtheorem*{theorem B}{Theorem B}

\theoremstyle{definition}
\newtheorem{definition}[theorem]{Definition}
\newtheorem{example}[theorem]{Example}
\newtheorem*{definition A}{Definition A}
\newtheorem*{definition B}{Definition B}

\theoremstyle{remark}
\newtheorem{remark}[theorem]{Remark}

\numberwithin{equation}{section}
\numberwithin{figure}{section}

\setlist[itemize]{topsep=0pt}

\title{Combination theorems for geometrically finite convergence
  groups}

\author{Alec Traaseth}
\address{Department of Mathematics, University of Virginia, Charlottesville, VA 22902, USA
}
\email{at3kk@virginia.edu}

\author{Theodore Weisman}
\address{Department of Mathematics, University of Michigan, Ann Arbor,
  MI 48109, USA
}
\email{tjwei@umich.edu}

\begin{document}
\maketitle

\begin{abstract}
  We prove combination theorems in the spirit of Klein and Maskit in
  the context of discrete convergence groups acting geometrically
  finitely on their limit sets. As special cases, we obtain
  combination theorems for geometrically finite groups of isometries
  of Hadamard manifolds with pinched negative curvature, and for
  relatively quasi-convex subgroups of relatively hyperbolic groups.
\end{abstract}

\tableofcontents

\section{Introduction}

When a group $G$ acts by homeomorphisms on a compact metrizable space
$M$, we say $G$ is a \textit{discrete convergence group} if the
induced action on the space of distinct triples in $M$ is properly
discontinuous. Despite their simple definition, convergence groups
carry a considerable amount of structure. They arise naturally when
considering actions of isometry groups of negatively curved metric
spaces on their ideal boundaries, and provide a way to study discrete
subgroups of isometries of these spaces from the perspective of
topological dynamics.

The action of a convergence group $G$ on a space $M$ is
\textit{geometrically finite} if every point of the limit set of $G$
in $M$ is either a \textit{conical limit point} or \textit{bounded
  parabolic point} for the action. Work of Bowditch \cite{Bowditch12},
\cite{Bowditch95} implies that, if $M$ is the boundary of a proper
geodesic negatively curved metric space $X$, and the action of $G$ on
$M = \del X$ is induced by an isometric action on $X$, then this
definition essentially agrees with the usual definitions of geometrical
finiteness in this geometric context. 

In this paper, we prove Klein-Maskit-type combination theorems for
convergence groups acting geometrically finitely on their limit
sets. Our theorems give sufficient conditions for geometrically finite
subgroups of $\Homeo(M)$ to generate a geometrically finite
convergence group isomorphic to an amalgamated free product or HNN
extension of the original groups. Precisely, we prove the following
two theorems:

\begin{restatable}{theorem A}{theoremA}
  \label{thm:theorem_a}
  Let $G_1$ and $G_2$ be discrete convergence groups acting on a
  compact metrizable space $M$. Suppose that $J = G_1 \cap G_2$ is
  geometrically finite, and $G_1$ and $G_2$ are in AFP ping-pong
  position with respect to $J$. Let $G = \la G_1, G_2 \ra < \Homeo(M)$, and suppose $G$ acts as a convergence group. Then the following hold:
    \begin{enumerate}[label=(\roman*)]
        \item \label{item:AFP_amalgamation} $G = G_1 *_J G_2$.
        \item \label{item:AFP_discrete} $G$ is discrete.
        \item \label{item:AFP_parabolics} Elements of $G$ not
          conjugate into $G_1$ nor $G_2$ are loxodromic.
        \item \label{item:AFP_gf} $G$ is geometrically finite if and
          only if both $G_1$ and $G_2$ are geometrically finite.
    \end{enumerate}
\end{restatable}

\begin{restatable}{theorem B}{theoremB}
  \label{thm:theorem_b}
  Let $G_0$ be a discrete convergence group acting on a compact
  metrizable space $M$, and suppose that $J_1, J_{-1} < G_0$ are both
  geometrically finite. Let $G_1 = \la f \ra$ be an infinite
  cyclic discrete convergence group also acting on $M$, where $fJ_{-1}f^{-1} = J_1$ in $\Homeo(M)$. Suppose $G_0$ is in HNN ping-pong
  position with respect to $f, J_1$ and $J_{-1}$. Let $G = \la G_0, G_1 \ra < \Homeo(M)$, and suppose $G$ acts as a convergence group. Then the following hold:

  \begin{enumerate}[label=(\roman*)]
  \item\label{item:HNN_extension} $G = G_0 *_f$.
  \item\label{item:HNN_discrete} $G$ is discrete.
  \item\label{item:HNN_parabolics} Elements of $G$ not conjugate into
    $G_0$ are loxodromic.
  \item\label{item:HNN_gf} $G$ is geometrically finite if and only if
    $G_0$ is geometrically finite.
  \end{enumerate}
\end{restatable}

The exact definitions of ``AFP ping-pong position'' and ``HNN
ping-pong position'' are given at the beginning of
\Cref{sec:theorem_a} and \Cref{sec:theorem_b}, respectively. They are
versions of the ``ping-pong'' configuration of limit sets required by
Maskit's combination theorems for Kleinian groups (see
\cite{Maskit88}).

\subsection{Special case: $M$ is the boundary of a negatively curved
  metric space}

If $X$ is a proper geodesic metric space which is hyperbolic in the
sense of Gromov, then any discrete subgroup of $\Isom(X)$ acts on both
the Gromov boundary $\del X$ of $X$ and the compactification
$X \sqcup \del X$ as a discrete convergence group. In particular this
holds if $X$ is a Hadamard manifold with pinched negative curvature,
e.g. a rank-1 symmetric space of noncompact type.

Theorem A and Theorem B both apply in this situation, meaning they
imply combination theorems for geometrically finite subgroups of
rank-1 semisimple Lie groups. In particular, in the special case
$M = \partial \H^n_\R$, Theorem A recovers a result of Li-Ohshika-Wang
\cite{LOW09}, who proved a version of Maskit's combination theorem for
amalgamated free products of geometrically finite groups acting on
real hyperbolic space of any dimension.

In \cite{LOW15}, Li-Ohshika-Wang also proved a version of Maskit's HNN
extension theorem in arbitrary-dimensional real hyperbolic
space. Theorem B is \emph{not} strong enough to fully recover this
result, and both Theorem A and Theorem B fail to fully recover
Maskit's analogous combination theorems for Kleinian groups. The
reason is that we impose some additional hypotheses on the relative
positions of the limit sets of the subgroups we are combining. See
\Cref{rem:maskit_differences} for more detail.

\subsection{Combination theorems for relatively hyperbolic groups}

Any relatively hyperbolic group acts as a convergence group on its
Bowditch boundary, which means that Theorem A and Theorem B also
directly imply combination theorems for \emph{relatively quasi-convex}
subgroups of relatively hyperbolic groups.

A number of combination theorems along these lines can be found in the
literature; see for instance \cite{Gitik}, \cite{MartinezPedroza},
\cite{MPS12}, \cite{Yang12}, \cite{MM22}. The combination theorems
given by these papers are all \emph{virtual}: they provide conditions
guaranteeing that certain finite-index subgroups $G_1'$, $G_2'$ of
relatively quasi-convex subgroups $G_1$, $G_2$ generate a relatively
quasi-convex subgroup, isomorphic to an amalgam of $G_1'$ and
$G_2'$. In contrast, the results in this paper give explicit
conditions which can be used to verify that a particular pair of
relatively quasi-convex subgroups generates a relatively quasi-convex
amalgam.

We can additionally contrast the combination theorems in this paper
with \emph{abstract} combination theorems for relatively hyperbolic
groups, which provide conditions that guarantee that the fundamental
group of some graph of relatively hyperbolic groups is also relatively
hyperbolic. Combination theorems of this type were originally proved
for hyperbolic groups by Bestvina-Feighn \cite{BF92}, and various
generalizations have been given for relatively hyperbolic groups (see
\cite{Dahmani03}, \cite{Alibegovic05}, \cite{MR08}, \cite{Gautero16},
\cite{Tomar2022a}, \cite{Tomar2022b}).

We give special mention to the work of Dahmani \cite{Dahmani03} (later
generalized by Tomar \cite{Tomar2022a}, \cite{Tomar2022b}), because
the method of proof is particularly relevant to the situation we
encounter in this paper. Dahmani and Tomar prove that an amalgam of
relatively hyperbolic groups is relatively hyperbolic by giving an
abstract construction of an appropriate ``limit set'' for the amalgam
to act on. That is, their strategy is to directly construct the
Bowditch boundary of the amalgam, and then prove that the amalgam acts
geometrically finitely on this space. Indeed, an alternative approach
to the proof of the combination theorems given in this paper would be
to show that the ``limit sets'' constructed in \cite{Dahmani03},
\cite{Tomar2022a}, \cite{Tomar2022b} appear embedded in the space $M$
on which our subgroups all act. However, in this paper we prefer a
direct approach, which we feel is more self-contained and
straightforward.

\subsection{Possible generalizations: higher rank combination theorems}

Dey-Kapovich-Leeb \cite{DKL19} have previously proved combination
theorems for \emph{Anosov} subgroups of higher-rank Lie groups, along
the lines of the combination theorems for hyperbolic groups proved by
Gitik in \cite{Gitik}. Anosov subgroups are discrete subgroups of
higher-rank Lie groups which generalize the dynamical behavior of
convex cocompact groups in rank one.

More recently, Dey-Kapovich \cite{DK22}, \cite{DK23} have proved
sharper combination theorems for Anosov subgroups, giving Maskit-type
ping-pong criteria (analogous to the ones in this paper) which
guarantee that a pair of Anosov subgroups generates a larger Anosov
subgroup, isomorphic to an amalgam of the original groups. The
Dey-Kapovich results also apply in rank one, which means that they
imply Theorem A and Theorem B in the special case where $M$ is the
visual boundary of a rank-one symmetric space $X$ and the subgroups
generating the amalgam $G$ are all convex cocompact in $\Isom(X)$.

Unlike the Dey-Kapovich results, the theorems in the present paper do
not apply directly in the higher rank setting. However, at their core,
our arguments only involve the topological dynamics of an action by
homeomorphisms on some space $M$, and do not rely directly on
geometric properties of any metric space bounded by $M$. Consequently,
it is reasonable to believe that our methods could be adapted to
extend the work of Dey-Kapovich, and prove combination theorems for
discrete subgroups of higher-rank Lie groups with ``geometrically
finite'' dynamical behavior---for instance \textit{relative} Anosov
subgroups, or the \textit{extended geometrically finite} subgroups
considered by the second author in \cite{Weisman22}.

\subsection{Tools used in the proof}

Although almost all of the arguments in this paper are purely
topological, we do use some metric geometry to prove a key technical
result at the end of \Cref{sec:background}. The proof of this
proposition (\Cref{prop:quasiconvex_subgroups_nest}) uses the
following fact, due to Yaman \cite{yaman2004topological}: if $G$ is a
geometrically finite convergence group acting on $M$, then there is a
proper $\delta$-hyperbolic metric space $X$ where $G$ acts by
isometries, such that $\del X$ is equivariantly homeomorphic to the
limit set of $G$ in $M$.

We use the coarse geometry of the space $X$, together with the
convergence action of $G$ on $M$, to establish some further dynamical
properties of the action of $G$ on $M$. Specifically, we prove that,
when a subgroup $J < G$ is \textit{fully quasi-convex}, there is a way
to modify sequences $(g_k)$ in $G \setminus J$ by elements of $J$, so
that certain sequences of the form $(g_kx)$ for $x \in M$ do not
accumulate on the limit set of $J$. Once we have established this
dynamical fact, we apply the conclusion repeatedly while working in
$M$, and no longer need to reference the geometry of any metric space.

\subsection{Acknowledgements}

Both authors would like to pay special thanks to Sara Maloni, who
initially suggested this problem and provided valuable feedback at
every stage of the project. The first author was partially supported
by the NSF Graduate Research Fellowship under Grant No. 1842490, as
well as NSF grants DMS-1848346 and DMS-1839968. The second author was
supported by NSF grant DMS-2202770.

\section{Convergence group actions and geometrical finiteness}
\label{sec:background}

This section is mostly devoted to background related to convergence
group actions. We start by defining convergence groups and geometrical
finiteness in \Cref{sec:convergence_group_background} and
\Cref{sec:geometrically_finite_background}. In
\Cref{sec:relatively_hyperbolic_background} we recall some background
on relatively hyperbolic groups. At the end of this section, we state and prove a key proposition
(\Cref{prop:quasiconvex_subgroups_nest}) about relatively quasi-convex
subgroups of geometrically finite convergence groups.

\subsection{Convergence groups}
\label{sec:convergence_group_background}

We refer to \cite{Tukia94}, \cite{Tukia98}, \cite{Bowditch99} for
further background on the material in this section.

\begin{definition}
  \label{defn:convergence_group}
  Let $G$ be a group acting on a compact metrizable space $M$. We say
  the action is a \textit{convergence action} and call $G$ a
  \textit{convergence group} if, whenever $(g_k)$ is a sequence of
  pairwise distinct elements in $G$, we can take a subsequence so that
  one of the following two conditions is satisfied:
  \begin{enumerate}
  \item The sequence $(g_k)$ converges to a homeomorphism $g$ in the compact-open
    topology on $\Homeo(M)$.
  \item There are points $z_+, z_- \in M$ (not necessarily distinct)
    so that the maps $g_k|_{M \setminus \{z_-\}}$ converge to the
    constant map $z \mapsto z_+$ uniformly on compacts.
  \end{enumerate}
  If $G$ is a convergence group such that only the second condition occurs, we
  call $G$ a \textit{discrete convergence group} and the action a
  \textit{discrete convergence action}.
\end{definition}

\begin{remark}\label{rmk:discrete_char}
  Note that when $G < \Homeo(M)$ is a convergence group, $G$ is a
  discrete convergence group if and only if $G$ is a discrete subgroup
  of $\Homeo(M)$ with respect to the compact-open topology on
  $\Homeo(M)$. So, if $(g_k)$ is a divergent sequence (that is, a
  sequence which leaves every compact subset of $\Homeo(M)$) in a
  discrete convergence group $G$, then we can extract a subsequence so
  that the second condition above holds.
\end{remark}

When $M$ is a topological $n$-sphere, the definition of a convergence
group is due to Gehring and Martin \cite{GM87}, who observed that the
isometry group of $\H_\R^n$ always acts as a convergence group on
$\del \H_\R^n$. Gehring and Martin also showed (again when $M$ is an
$n$-sphere) that a group $G$ is a discrete convergence group if and
only if the induced action of $G$ on the space of distinct triples in
$M$ is properly discontinuous; later Bowditch \cite{Bowditch99}
observed that the same holds when $M$ is an arbitrary compact
Hausdorff space.

In the setting where $M$ is compact metrizable, convergence groups
were studied systematically by Tukia \cite{Tukia94}. In particular
Tukia showed that any group of isometries acting properly
discontinuously on a proper geodesic Gromov-hyperbolic metric space
$X$ acts as a discrete convergence group on both the boundary $\del X$
and the compactification $\overline{X} = X \sqcup \del X$ (see also
Freden \cite{Freden1995}).

\begin{definition}
  Following \cite{Tukia94}, if $(g_k)$ is a sequence in
  $G < \Homeo(M)$ such that the second condition of
  \Cref{defn:convergence_group} holds \emph{without} extracting a
  subsequence, then we say that $(g_k)$ is a \textit{convergence
    sequence}.

  In this case, the (uniquely defined) points $z_+$ and $z_-$ are
  respectively called the \textit{attracting point} and
  \textit{repelling point} of the sequence $(g_k)$. It follows
  immediately that if $(g_k)$ is a convergence sequence with
  attracting and repelling points $z_+$, $z_-$, then $(g_k^{-1})$ is
  also a convergence sequence, with attracting and repelling points
  $z_-, z_+$.
\end{definition}

When $G$ is a discrete convergence group, an \emph{arbitrary} sequence
of pairwise distinct elements in $G$ is not necessarily a convergence
sequence, but it always has a subsequence which is.


One easy consequence of the definitions above is the following:
\begin{proposition}\label{prop:repeller_complement}
  Let $(g_k)$ be a convergence sequence in a discrete convergence
  group $G$ acting on a compact metrizable space $M$ containing at
  least 3 points. If $U$ is any open neighborhood of the repelling
  point $z_-$ of $(g_k)$, then $(g_k\overline{U})$ converges to $M$ in
  the topology on closed subsets of $M$ induced by Hausdorff distance.
\end{proposition}
\begin{proof}
  If $\overline{U} = M$ the result is immediate, so assume that
  $M \setminus U$ is nonempty. Then, since $M \setminus U$ is a
  nonempty compact subset of $M \setminus \{z_-\}$, the set
  $g_k(M \setminus U)$ converges to a singleton $\{z_+\}$. So
  $g_k\overline{U}$ eventually contains every compact in the
  complement of $\{z_+\}$, and must converge to the closure of
  $M \setminus \{z_+\}$. In addition, since $M$ contains at least 3
  points, there are distinct points $x,y \in M$ so that $(g_kx)$ and
  $(g_ky)$ both converge to $z_+$. This implies $z_+$ is not an isolated
  point of $M$ and so the closure of $M \setminus \{z_+\}$ is $M$.
\end{proof}

The set of attracting points (or equivalently, the set of repelling
points) of sequences in a discrete convergence group $G$ acting on a
space $M$ is called the \textit{limit set} of $G$ in $M$, and is
denoted $\Lambda(G)$. The limit set is always a closed $G$-invariant
subset of $M$. In fact, if $G$ is neither finite nor virtually cyclic,
then the limit set of $G$ is the unique minimal nonempty closed
$G$-invariant subset of $M$.

The complement of $\Lambda(G)$ in $M$ is denoted $\Omega(G)$, and is
called the \textit{domain of discontinuity} for $G$ since (as in the
setting of Kleinian groups) it is the maximal open subset of $M$ on
which $G$ acts properly discontinuously. Recall that a group $G$ acts properly discontinuously on a space $X$ if for any compact $K \subset X$, the set $\{g \in G \; | \; gK \cap K \neq \es\}$ is finite.

We say that a discrete convergence group $G$ is \textit{elementary} if
$|\Lambda(G)| \le 2$; this turns out to be equivalent to asking for
$|\Lambda(G)|$ to be finite. When $G$ is non-elementary, one can also
view $\Lambda(G)$ as the set of accumulation points of any $G$-orbit
in $M$.

The classification of isometries in hyperbolic space also generalizes
to a classification of the elements of a group $G$ acting as a
convergence group on $M$:
\begin{proposition}[\cite{Tukia94}]
  \label{prop:convergence_classification_of_isometries}
  Let $G$ act as a convergence group on a compact metrizable space
  $M$. Every $g \in G$ satisfies exactly one of the following:
  \begin{itemize}
  \item The closure of the cyclic group $\la g \ra$ is compact in
    $\Homeo(M)$, in which case we say $g$ is elliptic.
  \item $g$ is not elliptic and $g$ fixes exactly one point in $M$, in
    which case we say $g$ is parabolic.
  \item $g$ is not elliptic and $g$ fixes exactly two points in $M$,
    in which case we say $g$ is loxodromic.
  \end{itemize}
  Moreover, if $g$ is parabolic or loxodromic, then $(g^n)$ is a
  convergence sequence, and the set of attracting and repelling points
  $\{z_\pm\}$ of $(g^n)$ is precisely the set of fixed points of $g$.
\end{proposition}

If $G$ is a discrete convergence group, then the elliptic elements of
$G$ are precisely those with finite order. The classification also
implies that if $G$ is a virtually cyclic discrete convergence group,
then $G$ is elementary (but note that the converse need not hold).

\subsection{Geometrical finiteness}
\label{sec:geometrically_finite_background}

When $X$ is a Hadamard manifold with pinched negative curvature, the
geometrically finite subgroups of $\Isom(X)$ are the subgroups $G$
such that the quotient $X/G$ is topologically tame in a precise
sense. Geometrical finiteness was originally defined in real
hyperbolic spaces of dimension 2 and 3, where the definition concerned
the existence of a well-behaved fundamental domain for the action of
$G$ on $X$. This definition proved to be unsatisfactory in hyperbolic
spaces of higher dimension and in other negatively curved Hadamard manifolds,
however.

In \cite{Bowditch95}, Bowditch gave several different definitions of
geometrical finiteness for groups of isometries of a Hadamard manifold
$X$ with pinched negative curvature, and proved that they are all
equivalent. One of Bowditch's definitions (definition GF5), based on
work of Beardon and Maskit \cite{BM74}, can be expressed entirely in
terms of the convergence action of $G$ on its limit set in $\del X$,
and therefore generalizes readily to the situation where $G$ is a
convergence group acting on an arbitrary compact metrizable space $M$.

Before giving the definition we recall some essential terminology:
\begin{definition}\label{def:conical_limit_points}
  Let $G$ be a discrete convergence group acting on a compact
  metrizable space $M$.
  \begin{enumerate}[label=\roman*)]
  \item A point $x \in \Lambda(G)$ is a \textit{conical limit point}
    if there is a sequence $(g_k)$ in $G$ of distinct elements such
    that for every $z \in M \setminus \{x\}$, the pair $(g_kx, g_kz)$
    stays inside a compact subset of $(M \times M) \setminus \Delta$,
    where $\Delta \subset M \times M$ is the diagonal subspace. We
    will call the sequence $(g_k)$ a \textit{conical limiting
      sequence} for the point $x$.
  \item A point $x \in \Lambda(G)$ is a \textit{parabolic point} if it
    is the fixed point of a parabolic isometry in $G$. A
    \emph{parabolic subgroup} of $G$ is the stabilizer in $G$ of a
    parabolic point in $\Lambda(G)$. A parabolic point $x$ is
    \textit{bounded} if
    $(\Lambda(G) \setminus \{x\}) / \text{Stab}_G(x)$ is compact.
  \end{enumerate}
\end{definition}

\begin{remark}\label{rmk:conical_limit_seq_char}
  Tukia \cite{Tukia98} showed that no point in $M$ can be both a
  parabolic point and a conical limit point. By using the convergence
  group condition and extracting subsequences, one can also see that a
  point $x \in M$ is a conical limit point if and only if there are
  distinct points $a, b \in M$ and a conical limiting sequence $(g_k)$
  in $G$ such that $(g_kx)$ converges to $a$ and $(g_ky)$ converges to
  $b$ for all $y \ne x$. The sequence $(g_k)$ is then a convergence
  sequence, with $z_+ = b$ and $z_- = x$.

  Furthermore, we could just as well ask that the defining condition
  for a conical limiting sequence holds only for
  $z \in \Lambda(G) \setminus \{x\}$, and then the discrete
  convergence dynamics imply this also holds in $\Omega(G)$.
\end{remark}

\begin{definition}
  \label{defn:geometrically_finite}
  Let $G$ be a discrete convergence group acting on a compact
  metrizable space $M$. We say that $G$ is \textit{geometrically
    finite} if every point of $\Lambda(G)$ is either a conical limit
  point or a bounded parabolic point.
\end{definition}

\begin{remark}\label{rmk:gf_discrepancy}
  Unfortunately, the standard definitions of ``geometrically finite''
  in the geometric and dynamical contexts do not exactly
  agree. According to the definitions in e.g. \cite{Bowditch12},
  \cite{Dahmani03}, a convergence group $G$ acting on $M$ is
  ``geometrically finite'' if every point of $M$ (not just of
  $\Lambda(G)$) is a conical limit point or bounded parabolic
  point. With this convention, if $X$ is a Hadamard manifold with
  pinched negative curvature, and $G$ is a geometrically finite
  subgroup of $\Isom(X)$ (according to the definitions in
  \cite{Bowditch93}, \cite{Bowditch95}), then the action of $G$ on
  $\del X$ is not a ``geometrically finite convergence action'' if
  $\Lambda(G)$ is a proper subset of $\del X$.

  In this paper, we adopt the convention that a convergence group
  acting on $M$ is geometrically finite if and only if it acts
  geometrically finitely (in the sense of \cite{Bowditch12},
  \cite{Dahmani03}) on its limit set in $M$. So for us, when $G$ acts
  by isometries on a hyperbolic space $X$, ``geometrically finite''
  means the same thing regardless of whether we consider the isometric
  action on $X$ or the induced action by homeomorphisms on $\del X$.
\end{remark}

We conclude this subsection with another simple but useful criterion
which can be used to guarantee that a point $x \in M$ is a conical
limit point.
\begin{lemma}
  \label{lem:conical_characterization}
  Let $G$ be a discrete convergence group acting on a compact
  metrizable space $M$. Let $Y$ be a subset of $M$ containing at least
  two points, let $K_1, K_2$ be disjoint compact subsets of $M$, and
  let $x \in M$. If there exists a sequence $(g_k)$ of pairwise
  distinct elements of $G$ such that for all $k$ we have
  $g_kx \in K_2$ and $g_kY \subset K_1$, then $x$ is a conical limit
  point for $G$.
\end{lemma}
\begin{proof}
  Since $G$ is a discrete convergence group we can extract a
  subsequence so that, for points $z_\pm \in M$, the sequence $(g_k)$
  converges in $\Homeo(M)$ to the constant map $z_+$ uniformly on
  compacts. In particular, for any $y \ne z_-$, $(g_ky)$ converges to
  $z_+$. Since $Y$ contains at least two points, it contains at least
  one point $y$ not equal to $z_-$. Then since
  $g_ky \in g_kY \subset K_1$ we must have $z_+ \in K_1$. Since
  $g_kx \in K_2$, $(g_kx)$ cannot converge to $z_+$, hence $x =
  z_-$. Then for any $y \in M$ with $y \ne x$, $(g_ky)$ converges to
  $z_+$. The characterization of conical limit points described in
  \Cref{rmk:conical_limit_seq_char} implies that $x$ is a conical
  limit point.
\end{proof}

\subsection{Relatively hyperbolic groups}
\label{sec:relatively_hyperbolic_background}

For most of this paper, we will only ever need to work with the
dynamical definition of geometrical finiteness given above. However,
our proof of one key technical lemma
(\Cref{prop:quasiconvex_subgroups_nest}) does rely on a geometric
interpretation of the definition, which is best understood via the
connection between geometrically finite groups and \emph{relative
  hyperbolicity}. We refer to \cite{Bowditch12}, \cite{Hruska10} for
further background on relatively hyperbolic groups.

The definition of geometrical finiteness we will use is given in
\Cref{prop:geometrical_finite_definitions} below. As in the classical
(Kleinian) case, the definition says that, if $G$ is a geometrically
finite convergence group acting on a compact metrizable space $M$,
then an appropriately defined ``convex core'' for the $G$-action has a
``thick-thin'' decomposition into a compact piece and some standard
``cusps.'' When $M$ is the boundary of a $\delta$-hyperbolic metric
space $X$, this ``convex core'' can be defined via the following. For
any closed subset $Z$ of $\del X$, we let $\join(Z)$ denote the union
of all bi-infinite geodesics in $X$ joining distinct points in $Z$.

\begin{proposition}[see e.g. \cite{Bowditch12}, section 5]
  \label{prop:convex_hulls}
  Suppose that $X$ is a proper geodesic $\delta$-hyperbolic metric
  space, and $Z \subset \del X$ is a closed subset containing at least
  two points. Then $\join(Z)$ (with the metric induced by $X$) is the
  image of a quasi-isometrically embedded proper geodesic metric
  space, and its ideal boundary is precisely $Z$.
\end{proposition}

When $G$ is a Kleinian group, $\join(\Lambda(G))$ is within uniformly
bounded Hausdorff distance of the convex hull of the limit set of $G$,
i.e. the minimal closed $G$-invariant convex subset of $\H^3_\R$ whose
closure in $\overline{\H^3_\R}$ contains $\Lambda(G)$. So in the
general setting, we can think of the quotient $\join(\Lambda(G))/G$ as
a ``convex core'' for $X/G$.

\begin{proposition}[see \cite{Bowditch12}, section 6]
  \label{prop:geometrical_finite_definitions}
  Let $X$ be a proper geodesic $\delta$-hyperbolic metric space and
  let $G$ be an infinite discrete subgroup of $\Isom(X)$. Then the
  following are equivalent:
  \begin{itemize}
  \item The induced action of $G$ on $\del X$ is geometrically finite
    in the sense of \Cref{defn:geometrically_finite}.
  \item There exists a $G$-invariant system of pairwise disjoint
    horoballs $\mathcal{B}$ in $X$, such that the stabilizer in $G$ of
    each $B \in \mathcal{B}$ is a parabolic subgroup, and $G$ acts
    cocompactly on the set
    \[
      C(G, \mathcal{B}) := \join(\Lambda(G)) \setminus \bigcup_{B \in
        \mathcal{B}} B.
    \]
  \end{itemize}
  Moreover, if $|\Lambda(G)| > 1$, then for \emph{any} $G$-invariant
  system of pairwise disjoint horoballs $\mathcal{B}$ in $X$, the
  action of $G$ on $C(G, \mathcal{B})$ is cocompact if and only if the
  set of centers of horoballs in $\mathcal{B}$ is precisely the set of
  parabolic points in $\Lambda(G)$.
\end{proposition}
\begin{proof}
  Since $G$ is infinite and discrete, $\Lambda(G)$ cannot be empty. If
  $|\Lambda(G)| = 1$, then the first bullet point is trivial because
  the unique point in $\Lambda(G)$ is trivially bounded parabolic, and
  the second bullet point is trivial because $\join(\Lambda(G))$ is
  empty. So assume $|\Lambda(G)| > 1$.

  The space $Y = \join(\Lambda(G))$ is a \textit{taut} hyperbolic
  metric space (i.e. every point in $Y$ lies within uniformly bounded
  distance of a bi-infinite geodesic in $Y$). Furthermore, horoballs in $Y$
  (which can be viewed as a proper geodesic hyperbolic metric space
  via \Cref{prop:convex_hulls}) are at a uniformly bounded Hausdorff
  distance away from horoballs in $X$ intersected with $Y$. The result now
  follows from Proposition 6.12 and Proposition 6.13 in
  \cite{Bowditch12} after replacing $X$ with $Y$.
\end{proof}

If $|\Lambda(G)| > 1$ in the situation above, then we say $G$ is a
\textit{relatively hyperbolic group}, and the stabilizers of horoballs
in $\mathcal{B}$ are called the \textit{peripheral subgroups}. We say
$G$ is \textit{hyperbolic relative to} the collection $\mathcal{P}$ of
peripheral subgroups. We also say that any countably infinite group
$G$ is hyperbolic relative to $\{G\}$, and that any finite group is
hyperbolic relative to an empty collection of peripheral subgroups.

In the special case where $X$ is taut and $\Lambda(G) = \del X$, we
say that $X$ is a \textit{cusped space} for the data of the relatively
hyperbolic group $G$ and the peripheral subgroups $\mathcal{P}$. If
$|\Lambda(G)| > 1$ we can always find a cusped space by replacing $X$
with $\join(\Lambda(G))$.

The cusped space is in general \emph{not} uniquely determined, even up
to quasi-isometry. However, its ideal boundary is a well-defined
$G$-space once the peripheral subgroups of $G$ have been specified
(see section 9 in \cite{Bowditch12}). This space is called the
\textit{Bowditch boundary} of $G$ and we denote it $\del G$ (the
notation ignores the dependence on $\mathcal{P}$). When
$\mathcal{P} = \{G\}$, then the Bowditch boundary of $G$ is defined to
be a singleton, and when $G$ is finite its Bowditch boundary is empty.

When $|\del G| \le 2$, then we say $G$ is \textit{elementary}. The
Bowditch boundary of a non-elementary relatively hyperbolic group is
always \textit{perfect}, i.e. it contains no isolated points. In
particular if $|\del G| \ge 3$, then $\del G$ is infinite.

A result of Yaman shows that the action of $G$ on its Bowditch
boundary can actually be used to completely recover the definition of
$G$ as a relatively hyperbolic group:
\begin{theorem}[\cite{yaman2004topological}]
  \label{thm:yaman_relhyp_theorem}
  Let $G$ be a discrete convergence group acting on a perfect compact
  metrizable space $M$. If every point of $M$ is either a conical
  limit point or a bounded parabolic point (equivalently, if $G$ is
  geometrically finite and $\Lambda(G) = M$), then there is a proper
  geodesic $\delta$-hyperbolic metric space $X$, an embedding
  $G \to \Isom(X)$, and a $G$-equivariant homeomorphism from $M$ to
  $\del X$.
\end{theorem}

The theorem implies in particular that a geometrically finite convergence group is exactly the same thing as a relatively
hyperbolic group.

\begin{remark}
  Some definitions of relative hyperbolicity explicitly require either
  the group $G$ or the peripheral subgroups in $\mathcal{P}$ to be
  finitely generated. We do not make this assumption in this paper,
  since both \Cref{prop:geometrical_finite_definitions} and
  \Cref{thm:yaman_relhyp_theorem} hold without it. Our setup does
  always force the groups in $\mathcal{P}$ to be infinite, since they
  are parabolic subgroups of a convergence group.
\end{remark}

\subsubsection{Accumulation in geometrically finite subgroups}

Yaman's theorem means that we can always understand a non-elementary
discrete convergence group $G$ which is geometrically finite in the sense of
\Cref{defn:geometrically_finite} using its isometric action on a
cusped space $X$. In the case $|\del G| = 0$ or $|\del G| = 2$, we can
also find a cusped space by taking $X$ to be either a point or a line;
if $|\del G| = 1$ and $G$ is finitely generated, then we can take the
cusped space to be a ``horoball'' modeled on $G$ (see \cite{GM08},
\cite{Hruska10}).

We take advantage of the existence of the cusped space to prove some
properties of subgroups of $G$ which act geometrically finitely on
$\Lambda(G)$. A convenient notation we will use here and many times
later whenever we have a group $G$ acting on $M$ is
\[
    H(U) = \bigcup_{g \in H} gU
\]
for some $U \subset M, H \subset G$. For the orbit of a point, we will just write $Hx$.

\begin{definition}
  \label{defn:quasiconvex}
  Let $G$ be a relatively hyperbolic group, with Bowditch boundary
  $\del G$. A subgroup $H \le G$ is \textit{relatively quasi-convex}
  if $H$ acts geometrically finitely on $\del G$ (i.e. if every point
  of $\Lambda(H) \subseteq \del G$ is either a conical limit point or
  a bounded parabolic point for the $H$-action).

  Following \cite{Dahmani03}, we say that a relatively quasi-convex
  subgroup $H$ is \textit{fully quasi-convex} if for all but finitely
  many left cosets $gH$, we have
  $gH(\Lambda(H)) \cap \Lambda(H) = \es$.
\end{definition}

Observe that, if $G$ is elementary, then any fully quasi-convex
subgroup of $G$ is either finite or has finite index in $G$.

\begin{lemma}
  \label{lem:quasiconvex_accumulation}
  Let $G$ be a non-elementary relatively hyperbolic group with
  associated cusped space $X = X(G)$, and let $H \le G$ be a fully
  quasi-convex subgroup of $G$.

  Fix $x \in X$, and suppose that $(g_k)$ is an infinite sequence in
  $G \setminus H$ such that
  \begin{equation}
    \label{eq:g_k_distance}
    d_X(g_k x, x) = d_X(g_k x, H x)
  \end{equation}
  for all $k$. Then no attracting point of $g_k$ in $\del X$ lies in
  $\Lambda(H)$.
\end{lemma}
\begin{proof}
  Suppose for a contradiction that $g_k$ has an attracting point
  $z \in \Lambda(H) \subset \del X$. It follows that $H$ is infinite,
  since $\Lambda(H)$ is nonempty. Since $G$ acts as a convergence
  group on both $\del X$ and $X \sqcup \del X$, we see that $(g_kx)$
  converges to $z$ in $X \sqcup \del X$.
  
  For each $k$, we let $c_k:[0, r_k] \to X$ be a geodesic ray in $X$
  from $x$ to $g_kx$; since $(g_k)$ is divergent we have
  $r_k \to \infty$. We may extend each $c_k$ to a map
  $[0, \infty) \to X$ by setting $c_k(t) = c_k(r_k)$ for all
  $t \ge r_k$. Up to subsequence, these maps converge uniformly on
  compacts to a geodesic ray $c_z:[0, \infty) \to X$, whose ideal
  endpoint must be $z$.

  By \Cref{prop:geometrical_finite_definitions}, there is a
  $G$-invariant family $\mathcal{B}_G$ of pairwise disjoint horoballs
  in $X$ such that the parabolic subgroups of $G$ are precisely the
  stabilizers of the horoballs in $\mathcal{B}_G$, and the quotient of
  \begin{equation}
    \label{eq:core_complement}
    C(G, \mathcal{B}_G) = X \setminus \bigcup_{B \in \mathcal{B}_G} B
  \end{equation}
  by the action of $G$ is compact. By shrinking the horoballs in
  $\mathcal{B}_G$ if necessary, we can also assume that
  $x \in C(G, \mathcal{B}_G)$.

  We claim that $z$ is the center of some horoball
  $B \in \mathcal{B}_G$. If $H$ is an infinite subgroup of a parabolic
  subgroup $P$ in $G$, this is immediate, because then the unique
  point in $\Lambda(H)$ is the center of the unique horoball in
  $\mathcal{B}_G$ fixed by $P$. Otherwise, $\Lambda(H)$ contains at
  least two points, and we can consider the space
  $\join(\Lambda(H)) \subset X$.

  Let $\mathcal{B}_H$ be the horoballs in $\mathcal{B}_G$ whose
  centers are parabolic points in $\Lambda(H)$. By
  \Cref{prop:geometrical_finite_definitions} again, $H$ acts
  cocompactly on the set
  \[
    C(H, \mathcal{B}_H) := \join(\Lambda(H)) \setminus \bigcup_{B \in
      \mathcal{B}_H}B.
  \]
  Since the endpoint of the geodesic $c_z$ lies in $\Lambda(H)$, there
  is some uniform $R > 0$ so that every point in the image of $c_z$
  lies within distance $R$ of $\join(\Lambda(H))$.

  Now, suppose that for arbitrarily large $t$, the point $c_z(t)$ lies
  in an open $R$-neighborhood of the set $C(H, \mathcal{B}_H)$. But
  then for some $k = k(t)$, the point $c_k(t)$ also lies in an
  $R$-neighborhood of $C(H, \mathcal{B}_H)$. Since $H$ acts
  cocompactly on $C(H, \mathcal{B}_H)$, this means that $c_k(t)$ is
  within uniform distance of $hx$ for some $h \in H$. But this
  contradicts assumption \eqref{eq:g_k_distance}.

  So, for all sufficiently large times $t$, $c_z(t)$ must lie in some
  horoball in $\mathcal{B}_H$. Since the horoballs in $\mathcal{B}_H$
  are pairwise disjoint, there is in fact a single horoball
  $B \in \mathcal{B}_H$ so that $c_z(t)$ is in the interior of $B$ for
  all large enough $t$. The center of this horoball must be $z$.

  Since $(c_k)$ converges to $c_z$, for all sufficiently large $k$,
  the geodesic $c_k$ enters $B$. However, since we have assumed
  $x \in C(G, \mathcal{B}_G)$, we know that
  $g_kx \in C(G, \mathcal{B}_G)$, and thus $c_k$ must also leave the
  horoball $B$ after it enters it. So, let $w_k$ denote the last point
  where $c_k$ leaves $B$. The distances $d_X(x, w_k)$ must tend to
  infinity as $k \to \infty$, since $c_z$ never leaves $B$. See
  \Cref{fig:horoball_geodesic}.

  \begin{figure}[h]
    \centering
    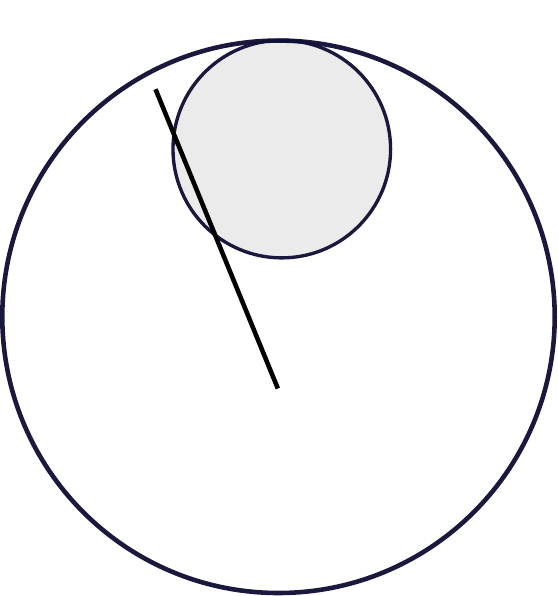
    \caption{Illustration for the proof of
      \Cref{lem:quasiconvex_accumulation}. The geodesic $c_k$ from $x$
      to $g_kx$ must enter $B$, and leave $B$ far from $x$.}
    \label{fig:horoball_geodesic}
  \end{figure}

  Since $c_k$ is a geodesic we know that
  $d_X(x, g_kx) = d_X(x, w_k) + d_X(w_k, g_kx)$. Then, because
  $d_X(x, w_k)$ tends to infinity, our assumption
  \eqref{eq:g_k_distance} implies that $d_X(Hx, w_k)$ tends to
  infinity as well. But, we also know that the stabilizer of $B$ in
  $G$ acts cocompactly on $\del B$. Then since $w_k \in \del B$, there
  is some constant $D > 0$ so that for every $k$, we have $s_k \in G$
  preserving $B$ such that $d_X(x, s_k^{-1}w_k) < D$, hence
  $d_X(s_kx, w_k) < D$. It follows that the elements in the sequence
  $(s_k)$ cannot lie in finitely many left cosets of $H$. However,
  since $s_k$ preserves $B$, each $s_k$ also fixes the point
  $z \in \Lambda(H)$, which contradicts the full quasi-convexity of
  $H$.
\end{proof}

The geometric statement of the lemma above has the following
(completely dynamical) consequence:
\begin{lemma}\label{lem:upgraded_technical_lemma}
  Let $G$ be a relatively hyperbolic group with Bowditch boundary
  $\del G$, and let $J_1, J_2$ be fully quasi-convex subgroups of $G$.

  For any sequence $(g_k)$ in $G$, there exists
  $j_k \in J_1, j_k' \in J_2$ such that the sequence $(j_kg_kj_k')$
  has no attracting points in $\Lambda(J_1) \subset \del G$ and no
  repelling points in $\Lambda(J_2) \subset \del G$.
\end{lemma}
\begin{proof}
  If $g_k \in J_1 \cup J_2$, then we can choose $j_k$ and $j_k'$ so
  that $j_kg_kj_k'$ is the the identity. A bounded sequence has no
  attracting or repelling points. So, we may assume
  $g_k \in G \setminus (J_1 \cup J_2)$ for all $k$.
  
  If $G$ is elementary, then $J_1$ and $J_2$ are both either finite or
  finite-index subgroups of $G$. In this case the result is immediate,
  so we can assume $G$ is non-elementary and let $X$ be a cusped space
  for $G$. Fix $x \in X$. For each $k$, we choose $j_k \in J_1$,
  $j_k' \in J_2$ so that
  \[
    d_X(g_k (J_2x), J_1x) = d_X(g_kj_k'x, j_k^{-1}x).
  \]
  We know such $j_k, j_k'$ exist because $J_i x$ are discrete
  subsets of $X$ for $i = 1, 2$. Let $g_k' = j_kg_kj_k'$. We will show
  that $g_k'$ has no repelling points in $\Lambda(J_2)$; the argument
  that $g_k'$ has no attracting points in $\Lambda(J_1)$ is completely
  symmetric, after replacing $g_k'$ with its inverse.

  Since $jJ_ix = J_ix$ for any $j \in J_i$, we know that for all $k$
  we have
  \[
    d_X(g_k'x, x) = d_X(g_kJ_2x, J_1x) = d_X(g_k'J_2x, J_1x).
  \]
  By definition, we know that
  \[
    d_X(g_k'J_2x, J_1x) \le d_X(g_k'J_2x, x) = d_X(J_2x, (g_k')^{-1}x),
  \]
  so combining this with the previous equality we conclude
  \[
    d_X(x, (g_k')^{-1}x) = d_X(g_k'x, x) \le d_X(J_2x, (g_k')^{-1}x)
  \]
  so in fact $d_X(x, (g_k')^{-1}x) = d_X(J_2x, (g_k')^{-1}x)$ for
  every $k$. Then \Cref{lem:quasiconvex_accumulation} implies that
  $((g_k')^{-1})$ has no attracting points in $\Lambda(J_2)$, or
  equivalently $(g_k')$ has no repelling points in $\Lambda(J_2)$.
\end{proof}

Our main application of these lemmas is the technical proposition
below. Roughly, this proposition tells us that in certain circumstances, it is
possible to strengthen the ``ping-pong'' combinatorics of
geometrically finite convergence groups. That is, the proposition
gives us a way to modify a ``ping-pong'' element $g \in \Homeo(M)$, so
that instead of nesting the closure of an open subset $U \subset M$
inside of another open subset $V \subset M$, $g$ takes the closure of
$U$ inside of a fixed compact subset $K \subset V$. This ``strong
nesting'' property will be useful throughout the paper.

\begin{proposition}
  \label{prop:quasiconvex_subgroups_nest}
  Let $G$ be a geometrically finite convergence group acting on a
  compact metrizable space $M$, let $H$ be a subgroup of $G$, and let
  $J_1, J_2 \le H$ be fully quasi-convex subgroups of $G$.

  Let $U_1, U_2$ be open subsets of $M$ such that, for $i \in \{1,2\}$,
  we have $J_i(U_i) = U_i$ and
  $\Lambda(H) \setminus \Lambda(J_i) \subset U_i$. Suppose that for
  every $g \in H \setminus J_2$, we have
  $g(M \setminus U_2) \subset U_1$.

  Then, there exists a compact set $K \subset U_1$ such that for all
  $g \in H \setminus J_2$, we can find $j \in J_1$ such that
  $jg(M \setminus U_2) \subset K$.
\end{proposition}
\begin{proof}
  Suppose that the claim does not hold. This means that we can find a
  sequence of group elements $(g_k)$ in $H \setminus J_2$ such that
  for any sequence $(j_k)$ in $J_1$, there is a sequence $(x_k)$ in
  $M \setminus U_2$ such that the sequence $(j_kg_kx_k)$ accumulates
  in $M \setminus U_1$.

  Fix this sequence $(g_k)$. \Cref{lem:upgraded_technical_lemma} gives
  a pair of sequences $(j_k)$ in $J_1$ and $(j_k')$ in $J_2$ so that
  any attracting points of the sequence $(g_k') = (j_kg_kj_k')$ do not
  lie in $\Lambda(J_2)$, and any repelling points do not lie in
  $\Lambda(J_1)$. Then, since $U_2$ is $J_2$-invariant, there is a
  sequence $(x_k)$ in $M \setminus U_2$ so that $(j_kg_kj_k'x_k)$
  accumulates in $M \setminus U_1$. After taking a subsequence, we may
  assume that $(j_kg_kj_k'x_k)$ has a unique limit
  $z \in M \setminus U_1$.

  Again using the fact that $U_1$ and $U_2$ are invariant under $J_1$
  and $J_2$ respectively, we know that for every $k$, we have
  $g_k'(M \setminus U_2) \subset U_1$. So, if only finitely many
  different elements appear in the sequence $(g_k')$, we can find a
  fixed compact set $K \subset U_1$ so that
  $g_k'(M \setminus U_2) \subset K$ for every $k$, hence
  $g_k'x_k \in K$ for every $k$. This is impossible if
  $g_k'x_k \to z \in M \setminus U_1$.

  So, we may extract a subsequence so that the elements in $(g_k')$
  are pairwise distinct. After taking a further subsequence, we can
  find a pair of points $z_+ \in M \setminus \Lambda(J_1)$,
  $z_- \in M \setminus \Lambda(J_2)$ so that $(g_k')$ converges
  uniformly to the constant map $z_+$, uniformly on compacts in
  $M \setminus \{z_-\}$. Both of $z_\pm$ lie in $\Lambda(H)$, so in
  fact $z_+ \in U_1$ and $z_- \in U_2$.

  Since $M \setminus U_2$ is closed, $x_k$ cannot accumulate on $z_-$,
  which means $(g_k'x_k)$ converges to $z_+ \in U_1$, which
  contradicts the fact that $g_k'x_k \to z$.
\end{proof}

\section{Combinatorial group theory: amalgamated free products}
\label{sec:combinatorial_group_theory_afp}

Our first main result deals with \textit{amalgamated free products},
so we set up the notation and basic facts here. Our reference
throughout is \cite{Maskit88}. As before, $M$ will continue to denote
a compact metrizable space, although the results in this section are purely set-theoretic. We further assume throughout this section
that $G_1, G_2$ are subgroups of $\Homeo(M)$, and $G_1 \cap G_2 = J$,
where $J$ is a proper subgroup of both $G_1$ and $G_2$. We let $G$
denote $\la G_1, G_2 \ra$, the subgroup generated by $G_1$ and $G_2$.

The following definition will be convenient
in this section as well as later in the paper:

\begin{definition}\label{def:precisely_invariant}
    We say a subset $U \subset M$ is \textit{precisely invariant} under $J$ in $G$ if $U$ is $J$-invariant, and for every $g \in G \setminus J$, we have $gU \cap U = \es$.

    More generally, given subgroups $J_1,\cdots, J_n < G$, we say a tuple of subsets $(U_1, \cdots, U_n)$ is \textit{precisely invariant} under $(J_1, \cdots, J_n)$ in $G$ if each $U_i$ is precisely invariant under $J_i$ in $G$, and if for $i \neq j$ and for every $g \in G$, we have $gU_i \cap U_j = \es$.
\end{definition}

Given a word $g = g_1 \cdots g_n$ in the elements of $G_1$ and $G_2$,
we call $g$ a \textit{normal form} when the elements $g_i$ alternate
between $G_1 \setminus J$ and $G_2 \setminus J$. We say two normal
forms $g = g_1 \cdots g_n$ and $h = h_1 \cdots h_n$ are
\textit{equivalent} if $g$ can be obtained from $h$ by inserting
finitely many words of the form $jj^{-1}$ for $j \in J$. We set
\[
  G_1 *_J G_2 = J \cup \{\text{equivalence classes of normal forms}\}.
\]
We have a group operation on $G_1 *_J G_2$ given by concatenation,
which is well-defined on normal forms up to equivalence. The abstract
group $G_1 *_J G_2$ is called the \textit{free product of $G_1$ and
  $G_2$ amalgamated over $J$}.

The normal form $g = g_1 \cdots g_n$ is called an
\textit{$(i,j)$-form} if $g_1 \in G_i$ and $g_n \in G_j$. The
\textit{length} of the normal form is defined as $|g| = n$. By
convention, we will say that elements of $J$ have length 0. Note that
if $g$ is an $(i,j)$-form, then its formal inverse $g^{-1}$ is a
$(j,i)$-form.

There is group homomorphism
\begin{align*}
    \varphi : G_1 *_J G_2 &\to G\\
    g_1 \cdots g_n &\mapsto g_1 \circ \cdots \circ g_n,
\end{align*}
where on the right we are just composing the corresponding elements in
$\Homeo(M)$. This map is always surjective, but its kernel need not be
trivial. When $\varphi$ is an isomorphism, we will abuse notation and
leave it implicit, writing $G = G_1 *_J G_2$; then we can view
elements of the subgroup $G$ as (equivalence classes of) normal forms
in the abstract amalgamated free product $G_1 *_J G_2$.

Using a ping-pong technique (\Cref{prop:interactive_pair} below), we
can give a sufficient condition which guarantees that $\varphi$ is
actually an isomorphism.

\begin{definition}\label{def:interactive_pair}
  A pair of disjoint nonempty $J$-invariant sets $U_1, U_2 \subset M$
  is called an \textit{interactive pair} for $G_1$ and $G_2$ if for
  every $g \in G_i \setminus J$, we have $gU_i \subset U_{3-i}$.
    
  If, in addition, $gU_i \subset U_{3-i}$ is a proper inclusion for
  every $g \in G_i \setminus J$ for at least one of $i \in \{1, 2\}$,
  then we call $(U_1,U_2)$ a \textit{proper interactive pair}.
\end{definition}

\begin{remark}
  Maskit's convention is to call an interactive pair $U_1, U_2$
  \textit{proper} if the $G_i$-translates of $U_i$ do not cover
  $U_{3-i}$ for at least one $i \in \{1, 2\}$. Our assumption is
  slightly weaker, but does not change any of the standard arguments.
\end{remark}

It is immediate that if $(U_1, U_2)$ is an interactive pair, then
$U_i$ is precisely invariant under $J$ in $G_i$ for $i = 1,2$.

We observe the following:
\begin{proposition}\label{prop:interactive_pair_infinite}
  If $(U_1, U_2)$ is a proper interactive pair for $G_1$ and $G_2$,
  then both $U_1$ and $U_2$ are infinite sets.
\end{proposition}
\begin{proof}
  Since $J$ is a proper subgroup of $G_i$ for $i = 1, 2$, there is at
  least one element $g_1 \in G_1 \setminus J$ and at least one element
  $g_2 \in G_2 \setminus J$. We know that at least one inclusion
  $g_1U_1 \subset U_2$ or $g_2U_2 \subset U_1$ is proper, so
  $g_2g_1U_1$ is a proper subset of $U_1$. Therefore $U_1$ is
  infinite, and since $g_1U_1 \subset U_2$, so is $U_2$.
\end{proof}

Via the map $\varphi$, normal forms in $G_1 *_J G_2$ act in a
``ping-pong'' manner on the sets in an interactive pair.

\begin{lemma}[\cite{Maskit88} VII.A.9]\label{lem:interactive_pair_dynamics}
  Suppose we have an
  interactive pair $(U_1,U_2)$. Then if $g \in G_1 *_J G_2$ is an
  $(i,j)$-form, we have $\varphi(g)U_j \subset U_{3-i}$. Further, this
  inclusion is proper if $(U_1,U_2)$ is proper and $|g| \geq 2$.
\end{lemma}

The lemma can be proved via a straightforward combinatorial argument;
see the reference for details. To illustrate the idea, suppose the
$G_1$-translates of $U_1$ are all properly contained in $U_2$, and
that $g$ has length 2. If $g = g_1g_2$ is a $(2,1)$-form, then
$g_2(U_1) \subset U_2$ is already proper, and hence
$\varphi(g)U_1 \subset U_1$ is also a proper inclusion. If $g$ is a
$(1,2)$-form, then $g_2U_2 \subset U_1$ need not be a proper
inclusion, but then applying $g_1$ will cause the next inclusion
$\varphi(g)U_2 = g_1 g_2U_2 \subset U_2$ to be proper.

\begin{proposition}[Ping-pong for amalgamated free products; see \cite{Maskit88} VII.A.10]\label{prop:interactive_pair}
  Suppose $(U_1, U_2)$ is a proper interactive pair for $G_1$ and $G_2$. Set
  $G = \la G_1, G_2 \ra$. Then $G = G_1 *_J G_2$.
\end{proposition}
\begin{proof}
  We will show the surjective group homomorphism
  $\varphi : G_1 *_J G_2 \to G$ has trivial kernel. The only length 0
  element sent to the identity is the identity, and length 1 elements
  are all nontrivial in $G_1$ or $G_2$, so it suffices to show
  $\varphi(g) \neq 1$ when $|g| \geq 2$. Suppose $g$ is an
  $(i,j)$-form. We now note that because we have a proper interactive
  pair, $\varphi(g)U_j \subset U_{3-i}$ is a proper inclusion by
  \Cref{lem:interactive_pair_dynamics}, and so $\varphi(g)$ cannot be
  the identity. The result follows.
\end{proof}

\section{Theorem A}
\label{sec:theorem_a}

We now introduce the main definition for Theorem A.

\begin{definition A}[AFP ping-pong position]
  Let $G_1$ and $G_2$ act as discrete convergence groups on a compact
  metrizable space $M$, and suppose that $G_1 \cap G_2 = J$ is a
  geometrically finite group distinct from both $G_1$ and $G_2$. We
  say $G_1$ and $G_2$ are in \textit{AFP ping-pong position} (with
  respect to $J$) if there exist closed sets $B_1, B_2 \subset M$ with
  nonempty disjoint interiors satisfying the following:
  \begin{enumerate}
  \item \label{item:AFP_J-invariance} For $i \in \{1,2\}$, $B_i$ is $J$-invariant.
  \item \label{item:AFP_pingpong} For $i \in \{1, 2\}$, and for each $g \in G_i \setminus J$, $gB_i \subset \Int(B_{3-i})$.
  \item \label{item:AFP_technical} For $i \in \{1, 2\}$, $\Lambda(G_i) \setminus \Lambda(J) \subset \Int(B_{3-i})$.
  \end{enumerate}
\end{definition A}

The definition above for the most part mimics the setup in Maskit's
original combination theorem for amalgamated free products of Kleinian
groups. It may be helpful to consider the following concrete example.

\begin{example}

\begin{figure*}[ht!]
    \centering
    \def\svgwidth{14.7 cm}
    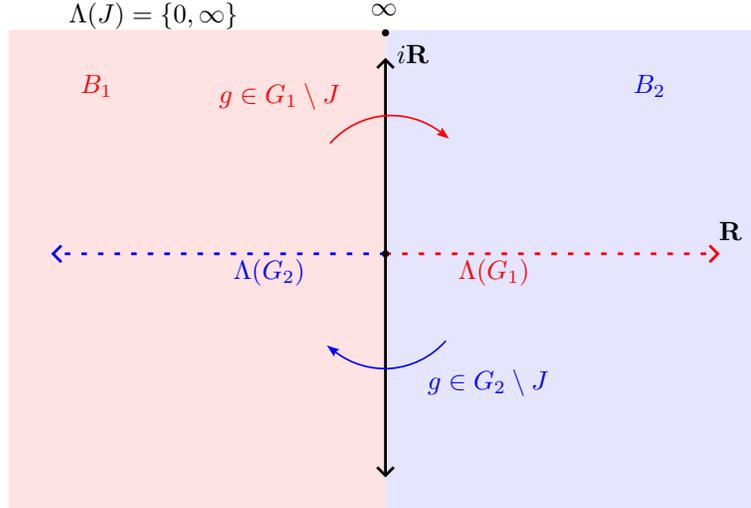
    \caption{Illustration for the example. The limit sets $\Lambda(G_i)$ are Cantor sets.}
    \label{fig:AFP_example}
\end{figure*}

Let $X = \H^3_\R$, and let $M$ be the visual boundary $\del \H^3_\R$,
viewed as the one-point compactification $\widehat{\C}$ of $\C$. We
let $G$ be a Fuchsian genus 2 surface group, embedded in
$\PSL_2(\C) \cong \Isom(\H^3)$ via the inclusion
$\PSL_2(\R) < \PSL_2(\C)$.

Then $G$ has the presentation
$\la a, b, c, d \; | \; [a,b][c,d] = 1 \ra$. Set
$j = [a,b] = [c,d]^{-1}$. We then take $G_1 = \la a, b \ra \cong F_2$,
and $G_2 = \la c, d \ra \cong F_2$, the free group on 2 letters, and
$J = G_1 \cap G_2 = \la j \ra \cong \Z$. We can arrange our generators
so $\Lambda(G_1) \subset \R_{\geq 0} \cup \{\infty\}$ and
$\Lambda(G_2) \subset \R_{\leq 0} \cup \{\infty\}$, where
$\Lambda(J) = \{0, \infty\}$. These will be the fixed points of
$j = (x \mapsto \lambda x)$ where $\lambda > 1$. See
\Cref{fig:AFP_example}.
    
This is precisely the picture one gets when gluing the sides of an
octagon in $\H^2_\R$ to form a surface of genus 2, and then
isometrically embedding this picture into the standard $\H^2_\R$
sitting inside $\H^3_\R$ whose boundary is $\R \cup \{\infty\}$. The
limit set of the surface group coincides with $\del \H^2_\R$, and then
the octagon with sides identified appears once in each connected
component of $\del \H^3_\R \setminus \del \H^2_\R$. Our $J$-invariant
sets $B_1$ and $B_2$ are then the closed left and right closed
half-planes respectively, including $\infty$, and their common
intersection is $i\R \cup \infty$. If we identify $\widehat{C}$ with
$S^2$, then $B_1$ and $B_2$ are complementary hemispheres.
\end{example}

We recall the main result of this section:

\theoremA*

\begin{remark}
  \label{rem:maskit_differences}
  The hypotheses for Theorem A are different from the hypotheses for
  Maskit's original combination theorems in $\H^3_\R$ in two
  respects. First, Maskit insists that the sets $B_1, B_2$ in
  Definition A are topological balls in $M = \del \H^3_\R$, satisfying
  $\del B_1 = \del B_2$ and $B_1 \cup B_2 = M$. This requirement is
  unnatural in our setting, since $M$ may not even be a manifold, and
  it is not needed in any of our arguments.

  Second, and more significantly, Maskit's version of condition
  \eqref{item:AFP_pingpong} in Definition A is weaker than what we
  have given here. Our condition implies in particular that if
  $g \in G_i \setminus J$, then $gB_i \cap B_i = \es$. This means that
  if $P < J$ is a maximal parabolic subgroup in $J$, then $P$ must
  also be a maximal parabolic subgroup in $G_i$.

  Maskit's original statement in $\H^3_\R$ allows $gB_i$ to
  intersect $B_i$ in limit points of $J$, which means his theorem
  allows for amalgamations along subgroups $J < G_i$ whose parabolic
  subgroups are \emph{not} maximal in $G_i$. This means our theorem is
  not strong enough to recover Maskit's original result in the case
  $M = \del \H_\R^3$. However, most of the examples constructed in Maskit's
  book satisfy the stronger hypothesis we have given above.
\end{remark}

Below, we give a quick proof of the first three parts of Theorem
A. The arguments are standard, but we provide them for convenience.

\begin{proof}[Proof of \ref{item:AFP_amalgamation} -
  \ref{item:AFP_parabolics} in Theorem A]
  \ref{item:AFP_amalgamation} Let $B_1, B_2$ be the closed subsets of
  $M$ from Definition A. We note that since $gB_i \subset \Int(B_{3-i})$ for any $g \in G_i \setminus J$, it follows that $g\Int(B_i) \subset \Int(B_i)$ is a proper inclusion for every $g \in G_i \setminus J$. Hence $(\Int(B_1), \Int(B_2))$
  form a proper interactive pair for $G_1$ and $G_2$ by conditions
  \eqref{item:AFP_J-invariance} and \eqref{item:AFP_pingpong}, so we
  are done by \Cref{prop:interactive_pair}.
    
  \ref{item:AFP_discrete} It suffices to show no sequence in $G$
  accumulates at the identity. Let $(g_k)$ be a sequence of
  distinct elements in $G$. Since $G_1, G_2$ are discrete we can assume
  $|g_k| > 1$. If the length of $g_k$ is odd, then $g_k$ maps one of
  the sets $B_1, B_2$ into the interior of the other and hence is far
  from the identity, so assume the lengths are all even. Without loss
  of generality, we may assume every $g_k$ is a $(2,1)$-form. We have
  $g_kB_1 \subset \Int(B_1)$ for every $k$.

  Suppose for a contradiction that $(g_k)$ converges to the
  identity. Then $g_kB_1$ converges to $B_1$. Write $g_k = h_k g_k'$
  where $|g_k'| = |g_k| - 1$ and $h_k \in G_2 \setminus J$. Then
  $g_kB_1 \subset h_kB_2 \subset \Int(B_1)$ for every $k$ since
  $g_k'B_1 \subset B_2$. It now also follows that $(h_kB_2)$ converges
  to $B_1$. Now, in general, when $g,h \in G_2 \setminus J$, we will
  have $gB_2$ and $hB_2$ either disjoint or equal. Indeed, if
  $gB_2 \cap hB_2 \neq \es$, then $h^{-1}g$ sends a point in $B_2$
  back into $B_2$, hence $h^{-1}g = j \in J$. Then
  $gB_2 = hjB_2 = hB_2$ as desired.

  Since $h_1B_2 \subset \Int(B_1)$ has nonempty interior and
  $h_kB_2 \subset \Int(B_1)$ converges to $B_1$, it follows that for
  some fixed large $k$, we will have $h_kB_2 \cap h_1B_2 \neq \es$, and
  also $h_kB_2 \neq h_1B_2$. This gives our contradiction, so we
  conclude $G$ is discrete.

  \ref{item:AFP_parabolics} Assume $g \in G$ is not conjugate into
  $G_1$ nor $G_2$. Take $g$ to have minimal length in its conjugacy
  class. If $g$ is an $(i,i)$-form (that is, $|g|$ is odd) then we can
  conjugate by an element of $G_i$ to reduce its length, hence $g$ has
  even length. Without loss of generality suppose $g$ is a
  $(2,1)$-form. Since $g^nB_1 \subset \Int(B_1)$ is a proper inclusion
  for every $n$, we see that $g$ has infinite order, hence is
  parabolic or loxodromic since $G$ is discrete. At least one fixed
  point of $g$ is an attracting point $z_+$ for the convergence
  sequence $(g^n)$ (see
  \Cref{prop:convergence_classification_of_isometries}). Since $B_1$
  has nonempty interior, there is some $w \in B_1$ so that
  $g^nw \to z_+$. But for every $n \ge 1$, the set $g^nB_1$ is a
  subset of the fixed compact $gB_1 \subset \Int(B_1)$, so we must
  have $z_+ \in \Int(B_1)$. An identical argument applied to $g^{-1}$
  (a $(1,2)$-form) gives a fixed point for $g$ in $\Int(B_2)$, hence
  $g$ is loxodromic by
  \Cref{prop:convergence_classification_of_isometries}.
\end{proof}

\subsection{Limit sets of amalgamated free products}

The rest of the section is devoted to the proof of part
\ref{item:AFP_gf} of Theorem A, so for the rest of the section, we fix
groups $G_1, G_2, J, G$ and sets $B_1, B_2 \subset M$ satisfying the
conditions of Definition A. We will prove each direction of the
theorem separately, but we start by making some general observations
about the positioning of the limit sets of subgroups of $G$.

\begin{proposition}\label{prop:AFP_basic_properties}
  Each of the following holds.
    \begin{enumerate}[label=(\roman*)]
        \item\label{item:AFP_prop_i} $\Lambda(J) \subset \del B_1 \cap
          \del B_2$. In particular, if $J$ is infinite, then $\del B_1
          \cap \del B_2$ is nonempty.
        \item\label{item:AFP_prop_ii} For $i \in \{1, 2\}$, $\Lambda(G_i)
          \subset B_{3-i}$.
        \item\label{item:AFP_prop_iii} For $i \in \{1,2\}$, and any
          $g \in G \setminus G_i$, we have
          $g(\Lambda(G_i)) \cap \Lambda(G_i) = \es$.
        \end{enumerate}
\end{proposition}
\begin{proof}
  \ref{item:AFP_prop_i} Note that since $J$ preserves the closed set
  $B_1$ and $\Int(B_1)$ is an infinite set by
  \Cref{prop:interactive_pair_infinite}, we have
  $\Lambda(J) \subset B_1$. Similarly, $\Lambda(J) \subset B_2$, hence
  $\Lambda(J) \subset B_1 \cap B_2 = \del B_1 \cap \del B_2$ since
  these sets have disjoint interiors.

  \ref{item:AFP_prop_ii} This is an immediate consequence of condition
  \eqref{item:AFP_technical} in Definition A along with
  \ref{item:AFP_prop_i} above.

  \ref{item:AFP_prop_iii} For concreteness, take $i = 1$, and let
  $g \in G \setminus G_1$. In particular $g \notin J$, so $g$ has a
  normal form with positive length. We can always find some
  $h, h' \in G_1$ so that $g' = hgh'$ is a $(1,2)$-form. Then,
  applying \ref{item:AFP_prop_ii}, we know that
  $g'\Lambda(G_1) \subset \Int(B_1)$ and so
  \[
    g'\Lambda(G_1) \cap \Lambda(G_1) = \es.
  \]
  Now, since $\Lambda(G_1)$ is invariant under $G_1$, we see that
  $g'\Lambda(G_1) = hg\Lambda(G_1)$, and therefore
  \[
    hg\Lambda(G_1) \cap \Lambda(G_1) = \es.
  \]
  But then
  $h^{-1}(hg\Lambda(G_1) \cap \Lambda(G_1)) = g\Lambda(G_1) \cap
  h^{-1}\Lambda(G_1) = g\Lambda(G_1) \cap \Lambda(G_1)$ is empty as
  well.
\end{proof}

\subsection{AFP ping-pong and contraction}

Both directions of the proof of Theorem A rely crucially on a key
contraction property of the ping-pong action of $G$ on the sets $B_1$ and $B_2$, stated as \Cref{lem:AFP_contraction} below. This contraction
lemma gives a sufficient condition for a sequence of sets $(g_kB_i)$
to converge to a singleton in $M$.

The proof of the contraction lemma relies on an application of
\Cref{prop:quasiconvex_subgroups_nest} to the subgroups we are
currently considering. Recall that this proposition gives us control
over the topological behavior of the action of fully quasi-convex
subgroups on certain subsets of $M$. So, in order to apply the
proposition, we first need to check:
\begin{lemma}
  \label{lem:fully_quasiconvex}
  Let $H$ be one of $G, G_1$, or $G_2$. If $H$ is geometrically
  finite, then $J$ is a fully quasi-convex subgroup of $H$.
\end{lemma}
\begin{proof}
  We know $J$ is relatively quasi-convex since it is a geometrically
  finite subgroup of $M$, so we just need to prove that for all but
  finitely many $h \in H \setminus J$ we have
  $h\Lambda(J) \cap \Lambda(J) = \es$. In fact, we will see that this
  is true for \emph{all} $h \in H \setminus J$.

  First, if $H = G_i$ for $i = 1$ or $2$, by assumption we know that
  for any $h \in H \setminus J$ we have
  $h\Lambda(J) \subset hB_i \subset \Int(B_{3-i})$, hence
  $h\Lambda(J) \cap \Lambda(J) = \es$ by part \ref{item:AFP_prop_i} of
  \Cref{prop:AFP_basic_properties}. If $H = G$, then any
  $h \in H \setminus J$ is an $(i,j)$-form, so that
  $h\Lambda(J) \subset hB_j \subset \Int(B_{3-i})$ and again
  $h\Lambda(J) \cap \Lambda(J) = \es$.
\end{proof}

Now, we can specialize \Cref{prop:quasiconvex_subgroups_nest} to the
current setting.
\begin{lemma}\label{lem:AFP_technical}
  Suppose that either $G$ is geometrically finite, or both $G_1$ and
  $G_2$ are geometrically finite. For $i \in \{1,2\}$, there exists a
  compact $K_i \subset \Int(B_{3-i})$ so that for any
  $g \in G_i \setminus J$, there is $j \in J$ so that
  $jgB_i \subset K_i$.
\end{lemma}
\begin{proof}
  This follows directly from \Cref{prop:quasiconvex_subgroups_nest},
  taking the ambient geometrically finite group $G$ to be either $G$
  or $G_i$ for $i \in \{1,2\}$, $H$ to be $G_i$, $J_1 = J_2 = J$,
  $U_1$ to be $\Int(B_{3-i})$, and $U_2$ to be $M \setminus B_i$. By
  assumption we know that
  $\Lambda(G_i) \setminus \Lambda(J) \subset \Int(B_{3-i}) \subset M
  \setminus B_i$, so in fact
  $\Lambda(G_i) \setminus \Lambda(J) \subset U_1 \cap U_2$ and the
  hypotheses of the proposition are satisfied.
\end{proof}

Finally, we can establish the contraction property for sequences in
$G$.

\begin{lemma}[Contraction for amalgamated free products]\label{lem:AFP_contraction}
  Suppose that either $G$ is geometrically finite, or both $G_1$ and
  $G_2$ are geometrically finite. If $(h_k)$ is a sequence of
  $(i,j)$-forms (for fixed $i$ and $j$) lying in distinct left cosets
  of $J$, then, up to subsequence, $(h_kB_j)$ converges to a singleton
  $\{x\}$.
\end{lemma}
It is not hard to verify directly that the subgroup
$\{g \in G : gB_j = B_j\}$ is exactly $J$. So, asking for the sequence
of cosets $(h_kJ)$ to be pairwise distinct is equivalent to asking for
the sequence of translates $(h_kB_j)$ to be pairwise distinct.
\begin{proof}
  We first prove the following:
  \begin{claim*}
    There exists a compact subset $K \subset \Int(B_{3-j})$ and a
    sequence $(j_k)$ in $J$ such that $j_kh_k^{-1}B_i \subset K$ for
    all $k$.
  \end{claim*}
  To prove the claim, first observe that if $|h_k| = 1$ for every $k$,
  then $i = j$ and $h_k \in G_i \setminus J$ for all $k$. Then the
  claim follows directly from \Cref{lem:AFP_technical}. Otherwise,
  suppose that $|h_k| > 1$, and write a normal form for $h_k$:
  \[
    h_k = g_{k,1} \cdots g_{k,n}.
  \]
  Although $n$ can depend on $k$, we ignore this in the notation. The
  word $h_k^{-1} = g_{k,n}^{-1} \cdots g_{k,1}^{-1}$ is a
  $(j,i)$-form, and the word
  \[
    g_{k,n}h_k^{-1} = g_{k,n-1}^{-1} \cdots g_{k,1}^{-1}
  \]
  is a $(3-j, i)$-form. This means that
  $g_{k,n}h_k^{-1}B_i \subset B_j$. Then, we can apply
  \Cref{lem:AFP_technical} again to find a fixed compact
  $K \subset \Int(B_{3-j})$ and $j_k \in J$ so that
  $j_kg_{k,n}^{-1}B_j \subset K$ for every $k$, and therefore
  \[
    j_kh_k^{-1}B_i = j_kg_{k,n}^{-1}g_{k,n}h_k^{-1}B_i \subset
    j_kg_{k,n}^{-1}B_j \subset K.
  \]
  This proves the claim, so now we consider the sequence
  $(h_kj_k^{-1})$. Since the left cosets $h_kJ$ are all distinct, it
  follows that the sequence of group elements $(h_kj_k^{-1})$ is
  divergent in $G$, and therefore we can extract a convergence
  subsequence: we can find attracting and repelling points
  $z_+, z_- \in M$ so that $(h_kj_k^{-1}y)$ converges to $z_+$
  whenever $y \ne z_-$. Equivalently, $(j_kh_k^{-1}y)$ converges to
  $z_-$ whenever $y \ne z_+$.

  By \Cref{prop:interactive_pair_infinite}, the set $B_j$ is infinite,
  so there is at least one point $y \in B_j \setminus \{z_+\}$. Since
  $j_kh_k^{-1}B_j \subset K$, we must have $z_- \in K$. In particular,
  $z_-$ must lie in $\Int(B_{3-j})$, which means that $B_j$ is a
  compact subset of $M \setminus \{z_-\}$. Thus,
  $(h_kj_k^{-1}B_j) = (h_kB_j)$ converges to the singleton $\{z_+\}$ as
  desired.
\end{proof}

\subsection{Geometrical finiteness of the product}

We now turn to the proof of the implication ($G_1$ and $G_2$
geometrically finite) $\implies$ ($G$ geometrically finite), which is
one of the directions of Theorem A \ref{item:AFP_gf}.

The proof of this direction of the theorem relies on the fact that
limit points of $G$ fall into one of two classes: either they are
$G$-translates of limit points of $G_1$ or $G_2$, or else they are
limit points of sequences of $(i,j)$-forms in $G$ whose length tends
to infinity. The essential step in the proof is to show that any limit
point $x$ of the latter form can be ``coded'' by a sequence of nested
translates of $B_1$ or $B_2$.

Precisely, we prove the following:
\begin{proposition}[AFP coding for $G$-limit points]
  \label{prop:AFP_limit_points_have_codings}
  Suppose that $G_1$ and $G_2$ are geometrically finite, and let $x$
  be a point in
  $\Lambda(G) \setminus G(\Lambda(G_1) \cup \Lambda(G_2))$. Then there
  exists a sequence $(g_k)$ in $(G_1 \cup G_2) \setminus J$ so that
  for every $k$,
  \[
    h_k = g_1 \cdots g_k
  \]
  has length $k$, and if $g_k \in G_j$, then $x \in h_kB_j$.
\end{proposition}

To prove this proposition, we follow Maskit's strategy, and consider a
sequence of ``ping-pong'' sets in $M$, defined inductively as
follows. We let $T_0 = B_1 \cup B_2$. Then, for every $n > 0$, and
$i \in \{1,2\}$, we define
\[
  T_{n,i} = \bigcup_{g \in G_i \setminus J} g(B_i \cap T_{n-1}).
\]
Then we define
\[
  T_n = T_{n,1} \cup T_{n,2}.
\]
The set $T_1$ is just the union of the $G_1 \setminus J$ translates of
$B_1$ and the $G_2 \setminus J$ translates of $B_2$. More generally,
$T_n$ is the union of translates of $B_1$ by $(i,1)$-forms of length
$n$ and the translates of $B_2$ by $(i,2)$-forms of length $n$. See \Cref{fig:AFP_T} for a depiction of $T_1$ and $T_2$. We see that
these sets are decreasing, so let
\[
    T = \bigcap_{n=0}^\infty T_n.
\]We
will see that limit points of $G$ which are not translates of limit
points of $G_1$ nor $G_2$ are in $T$, which allows us to construct the
sequence given by the conclusion of the proposition above.

\begin{figure*}[ht!]
    \centering
    \def\svgwidth{14.7 cm}
    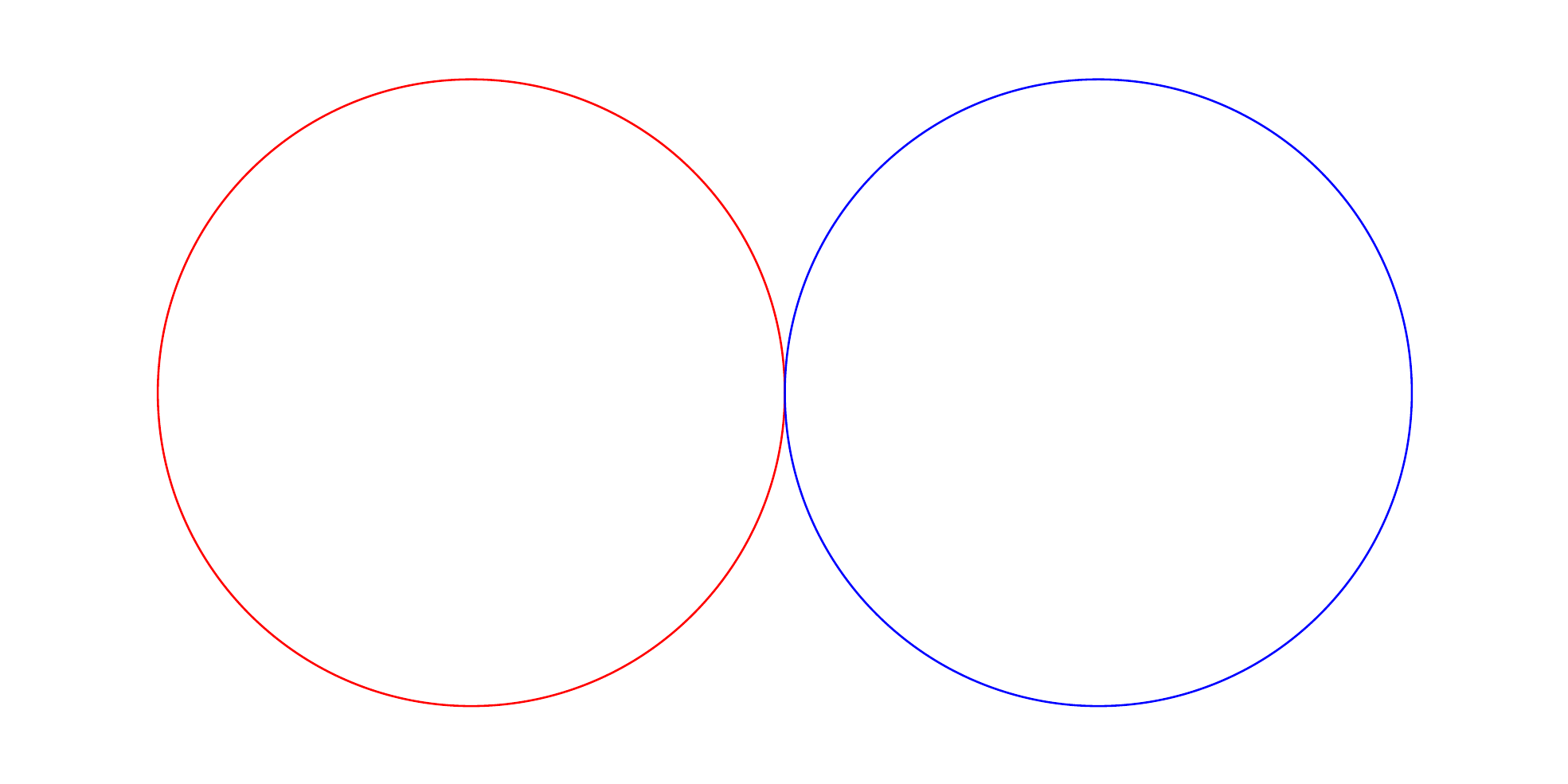
    \caption{Part of the sets $T_1$ and $T_2$.}
    \label{fig:AFP_T}
\end{figure*}

We observe:
\begin{lemma}
  \label{lem:AFP_T_G_invariant}
  The set $T$ is $G$-invariant and nonempty. In particular, since $G$
  is non-elementary, we have
  $\Lambda(G) \subset \overline{T} \subset B_1 \cup B_2$.
\end{lemma}
\begin{proof}
  We know $T$ is nonempty because it is the intersection of a
  decreasing sequence of nonempty subsets of the compact space
  $M$. The definition of $T_n$ implies that if $x \in T_n$ and
  $g \in G_1 \cup G_2$, then $g x \in T_{n-1}$. Inductively, we see
  that if $g \in G$ has $|g| = k$, and $x \in T_{n+k}$, then
  $gx \in T_n$. It follows that if $x \in T$ then $g x \in T$ for any
  $g \in G$.
\end{proof}

\begin{lemma}
  \label{lem:AFP_first_pingpong_inclusion}
  If $G_1$ and $G_2$ are geometrically finite, we have
  $\Lambda(G) \setminus (\Lambda(G_1) \cup \Lambda(G_2)) \subset T_1$.
\end{lemma}
\begin{proof}
  We will prove that if $y \in \Lambda(G) \setminus T_1$, then
  $y \in \Lambda(G_1) \cup \Lambda(G_2)$, so suppose that
  $y \in \Lambda(G)$ does not lie in $T_1$. Using the above lemma, we
  know $y \in B_1 \cup B_2$, so without loss of generality assume
  $y \in B_2$. Since $y$ is in the limit set of $G$, we can find a
  sequence $(g_k)$ in $G$ so that $(g_kw)$ converges to $y$ for all
  but a single point in $M$. If $y \in \Lambda(J)$ we are done, so we
  can assume that $g_k \notin J$ for infinitely many $k$.

  Then, after extracting a subsequence, we can assume that for every
  $k$, $g_k$ is an $(i,j)$-form for $i, j$ fixed, and then find
  $w \in B_j$ so that $g_kw \to y$.

  Since $g_k$ is an $(i,j)$-form and $w \in B_j$, we have
  $g_kw \in G_i(B_i)$ for every $k$. So, we may write $g_kw = g_k'z_k$
  for $g_k' \in G_i \setminus J$ and $z_k \in B_i$. Note that $g_k'$
  is just the first letter in the $(i,j)$-form $g_k$. In particular,
  we know $g_k'z_k \in T_1$ for every $k$, so $g_k'z_k$ is never equal
  to $y$. If, up to subsequence, there are only finitely many distinct
  translates $g_k'B_i$, then we would have
  $g_k'z_k \in \bigcup g_k'B_i$, a compact set in the complement of
  $T_1$, which contradicts the fact that $g_k'z_k \to y \in
  T_1$. Hence we may assume that the translates $g_k'B_i$ are all
  distinct, which means that the left cosets $g_k'J$ are all distinct.

  Now, \Cref{lem:AFP_contraction} implies that $(g_k'B_i)$ converges
  to a singleton. This singleton must be $y$ since $g_k'z_k \to
  y$. It follows that $g_k'z \to y$ for any $z \in B_i$, and since
  $B_i$ is an infinite set it follows that $y \in \Lambda(G_i)$ as
  desired.
\end{proof}

\begin{proof}[Proof of \Cref{prop:AFP_limit_points_have_codings}]
  We first claim that
  $\Lambda(G) \setminus G(\Lambda(G_1) \cup \Lambda(G_2))$ is a subset
  of $T$. So, fix $z \in \Lambda(G)$, and suppose $z \notin T$. We
  will show $z \in G(\Lambda(G_1) \cup \Lambda(G_2))$.

  By \Cref{lem:AFP_T_G_invariant} we know that
  $\Lambda(G) \subset B_1 \cup B_2 = T_0$, so there is some $n > 0$
  such that $z \in T_{n-1} \setminus T_n$. In particular, because
  $z \in T_{n-1}$, there is an $(i,j)$-form $g \in G$, with
  $|g| = n-1$, such that $gy = z$ for $y \in B_j$. We must have
  $y \notin T_1$, since otherwise we would have $y = hw$ for
  $w \in B_{3-j}$ and $h \in G_{3-j} \setminus J$, and then $z = ghw$
  would lie in $T_n$. Then, since $\Lambda(G)$ is $G$-invariant we see
  that $y \in \Lambda(G)$ but $y \notin T_1$, so by the previous lemma
  we have $y \in \Lambda(G_1) \cup \Lambda(G_2)$, hence
  $z \in G(\Lambda(G_1) \cup \Lambda(G_2))$.

  We have now seen that
  $\Lambda(G) \setminus G(\Lambda(G_1) \cup \Lambda(G_2))$ is a subset
  of $T$, so we just need to show that for any $x \in T$, there is a
  sequence of $(i,j)$-forms $(h_k)$ satisfying the conclusions of the
  proposition. We construct this sequence inductively. Take $h_0$ to
  be the identity. For $k > 0$, assume that $x \in h_{k-1}B_j$ for
  an $(i,j)$-form
  \[
    h_{k-1} = g_1 \cdots g_{k-1}.
  \]
  By \Cref{lem:AFP_T_G_invariant}, $T$ is $G$-invariant, so
  $h_{k-1}^{-1}x \in B_j \cap T$. In particular, $h_{k-1}^{-1}x$
  lies in $T_1 \cap B_j = T_{1,j}$, so there is some
  $g_k \in G_{3-j} \setminus J$ so that
  $h_{k-1}^{-1}x \in g_kB_{3-j}$. Then if $h_k$ is the
  $(i, 3-j)$-form
  \[
    g_1 \cdots g_k,
  \]
  we have $x \in h_k(B_{3-j})$ and $|h_k| = |h_{k-1}| + 1$, as
  required.
\end{proof}

The next step is to use the ``coding'' of limit points given by
\Cref{prop:AFP_limit_points_have_codings} to prove that there is a
conical limit sequence for every point in
$\Lambda(G) \setminus G(\Lambda(G_1) \cup \Lambda(G_2))$.

\begin{lemma}\label{lem:AFP_conical}
  If $G_1$ and $G_2$ are geometrically finite, every point in
  $\Lambda(G) \setminus G(\Lambda(G_1) \cup \Lambda(G_2))$ is a
  conical limit point for $G$.
\end{lemma}
\begin{proof}
  Let $x \in \Lambda(G) \setminus G(\Lambda(G_1) \cup
  \Lambda(G_2))$. We know $x \in B_1 \cup B_2$ from
  \Cref{lem:AFP_T_G_invariant}, so to simplify notation assume
  $x \in B_2$. We let $(g_k)$ be the sequence in
  $(G_1 \cup G_2) \setminus J$ from
  \Cref{prop:AFP_limit_points_have_codings}, so that, for every $k$,
  we have $|g_1 \cdots g_k| = k$ and if $g_k \in G_j$, then
  $x \in g_1 \cdots g_kB_j$.
  
  For each $k$, we let $h_k = g_1 \cdots g_{2k}$, so that $h_k$ is an
  $(i, j)$-form for fixed $i \ne j$. Since
  $h_kB_j \subset \Int(B_{3-i})$, and $x \in B_2$, we have $i = 1$ and
  thus $h_k$ is a $(1,2)$-form for every $k$. This means that
  $(g_{2k})$ is a sequence in $G_2 \setminus J$. So, using
  \Cref{lem:AFP_technical}, we find a fixed compact subset
  $K \subset \Int(B_1)$ and a sequence $(j_k)$ in $J$ so that
  $j_kg_{2k}^{-1}B_2 \subset K$.

  Consider the sequence $(f_k)$ given by $f_k = h_kj_k^{-1}$. Since
  $|f_k| \to \infty$, a subsequence of $(f_k^{-1})$ consists of
  pairwise distinct elements of $G$. Since $f_k^{-1}$ is a
  $(2,1)$-form, we know that
  \[
    f_k^{-1}B_1 = j_kg_{2k}^{-1} \cdots g_1^{-1}B_1 \subset K.
  \]
  On the other hand, by construction, we know that
  \[
    h_k^{-1}x = g_{2k}^{-1} \cdots g_1^{-1}x \in B_2.
  \]
  Since $B_2$ is $J$-invariant, we also see that
  $f_k^{-1}x = j_kh_k^{-1}x \in B_2$ for every $k$. By
  \Cref{prop:interactive_pair_infinite}, $\Int(B_1)$ is an infinite
  set. Then, since $B_2$ and $K$ are disjoint compact subsets of $M$,
  we can apply \Cref{lem:conical_characterization} (with
  $Y = \Int(B_1)$, $K_1 = K$, and $K_2 = B_2$) to complete the proof.
\end{proof}

Next we deal with parabolic points.

\begin{lemma}\label{lem:AFP_bounded_parabolic_G}
  If both $G_1$ and $G_2$ are geometrically finite, then every
  parabolic point of $G$ in $\Lambda(G_1) \cup \Lambda(G_2)$ is a
  bounded parabolic point for the action of $G$ on $\Lambda(G)$.
\end{lemma}
\begin{proof}
  Fix a parabolic point $p \in \Lambda(G_1)$, and let $P < G$ be the
  parabolic subgroup stabilizing $p$. We will show that there is a
  compact set $K \subset \Lambda(G) \setminus \{p\}$ so that
  $P(K) = \Lambda(G) \setminus \{p\}$, which implies the action is
  cocompact. The main idea here is to apply
  \Cref{prop:quasiconvex_subgroups_nest} to the parabolic subgroup
  $P$, which gives us a way to use elements of $P$ to position certain
  points in $M$ far away from $\Lambda(P) = \{p\}$. Our strategy is to
  decompose the set $\Lambda(G) \setminus \{p\}$ into pieces. We will
  show that every point in $\Lambda(G) \setminus \{p\}$ is either far
  away from $p$ to begin with, or else it is in a piece of
  $\Lambda(G) \setminus \{p\}$ which can be translated far away from
  $p$ using either \Cref{prop:quasiconvex_subgroups_nest} or the
  boundedness of $p$ in $\Lambda(G_1)$.
  
  We consider two cases. For both cases, in order to apply
  \Cref{prop:quasiconvex_subgroups_nest}, we need to know that $J$ and
  $P$ are fully quasi-convex subgroups of $G_1$; for $J$ this follows
  from \Cref{lem:fully_quasiconvex}, and for $P$ this is true because
  $P$ is exactly the stabilizer of its limit set
  $\{p\} \subset \Lambda(G_1)$ in $G_1$.
  \subsubsection*{Case 1:
    $p \in \Lambda(G_1) \setminus G_1(\Lambda(J))$}
  Using \Cref{lem:AFP_first_pingpong_inclusion}, we can see that every
  point in $\Lambda(G) \setminus \{p\}$ lies in one of the sets
  $\Lambda(G_1), \Lambda(G_2)$, or $T_1$. Since
  $\Lambda(G_2) \subset B_1$, and $T_1 \subset B_1 \cup B_2$, this
  means that every point in $\Lambda(G)$ lies in one of the sets
  \[
    L_1 = \Lambda(G_1), \quad L_2 = B_1, \quad L_3 = T_1 \cap B_2.
  \]
  Now, for each $i$, we will find a compact set
  $K_i \subset M \setminus \{p\}$ so that $P(K_i)$ contains
  $(\Lambda(G) \setminus \{p\}) \cap L_i$. Then we can define
  $K = (K_1 \cup K_2 \cup K_3) \cap \Lambda(G)$, so that
  $P(K) = \Lambda(G) \setminus \{p\}$.

  Since $p$ is a bounded parabolic point for the action of $G_1$ on
  $\Lambda(G_1)$, and $\Lambda(G_1)$ is locally compact, we already
  know that there is a compact
  $K_1 \subset \Lambda(G_1) \setminus \{p\}$ so that
  $P(K_1) = \Lambda(G_1) - \{p\}$. And, by part
  \ref{item:AFP_technical} of Definition A, we know $p \in \Int(B_2)$,
  so $B_1$ is already a compact subset of $M \setminus \{p\}$ and we
  can take $K_2 = B_1$. So, we just need to construct the compact set
  $K_3$.

  For this, we apply \Cref{prop:quasiconvex_subgroups_nest}, with
  $G = H = G_1$, $J_1 = P$, $J_2 = J$, $U_1 = M \setminus \{p\}$, and
  $U_2 = M \setminus B_1$. To verify that the hypotheses of the
  proposition are satisfied, we need to check that
  $gB_1 \subset M \setminus \{p\}$ for every $g \in G_1 \setminus
  J$. But, since $\Lambda(G_1)$ is $G_1$-invariant we can only have
  $p \in gB_1$ if $g^{-1}p \in B_1 \cap \Lambda(G_1) = \Lambda(J)$,
  which is impossible since we assume
  $p \in \Lambda(G_1) \setminus G_1(\Lambda(J))$.

  So, we know there is a compact subset $K' \subset M \setminus \{p\}$
  so that for any $g \in G_1 \setminus J$, we can find $h \in P$ so
  that $hgB_1 \subset K'$. But by definition, any $y \in T_1 \cap B_2$
  lies in $(G_1 \setminus J)(B_1)$, so we can take $K_3 = K'$ and we
  are done.

  \begin{figure}[h]
    \centering
    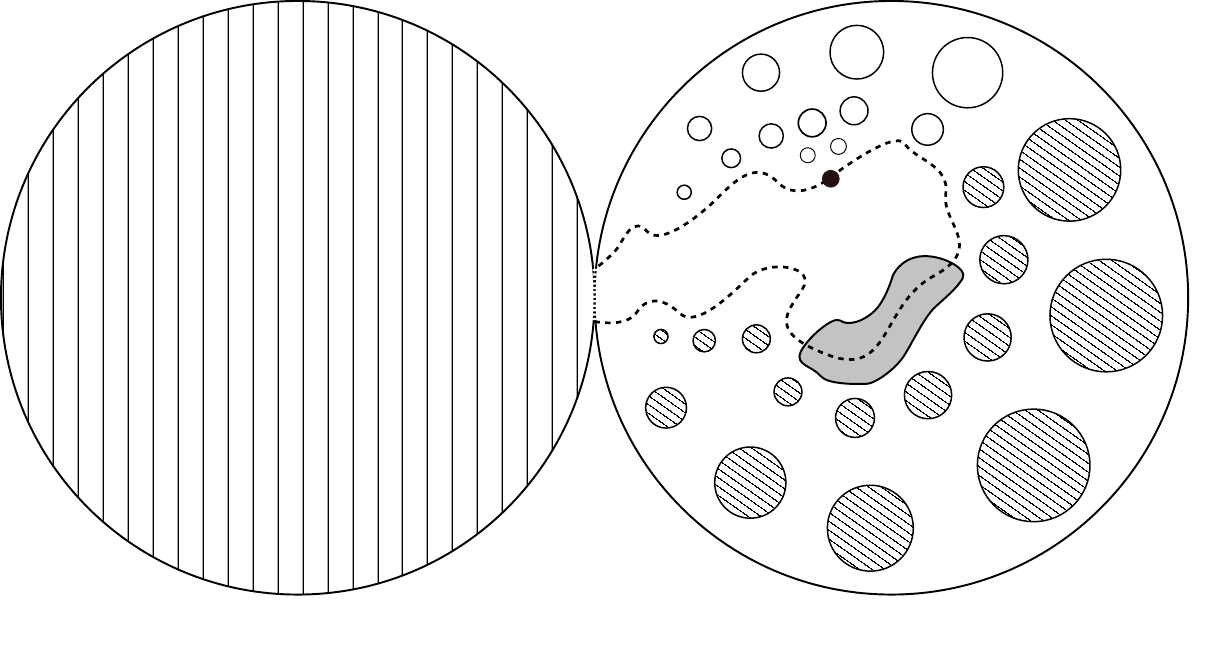
    \caption{The sets $K_1, K_2$, and $K_3$ proving that
      $p \in \Lambda(G_1)$ is a bounded parabolic point (Case 1).}
    \label{fig:afp_parabolic}
  \end{figure}
  
  \subsubsection*{Case 2:
    $p \in G_1(\Lambda(J))$}
  Since $G$ acts by homeomorphisms on $\Lambda(G)$ it suffices to
  consider the case $p \in \Lambda(J)$. For this case, we again use
  \Cref{lem:AFP_first_pingpong_inclusion} to see that every point in
  $\Lambda(G)$ lies in one of the three sets
  \[
    L_1 = \Lambda(G_1), \quad L_2 = \Lambda(G_2), \quad L_3 = T_1.
  \]
  As in the previous case, for each of these sets, we will find a
  compact set $K_i \subset M \setminus \{p\}$ so that $P(K_i)$
  contains $(\Lambda(G) \setminus \{p\}) \cap L_i$.

  For $i = 1,2$, as in Case 1, we can use the fact that $p$ is a
  bounded parabolic point for the $G_i$-action on $\Lambda(G_i)$, to
  find compact sets $K_i \subset \Lambda(G_i) \setminus \{p\}$ such
  that $P(K_i) = \Lambda(G_i) \setminus \{p\}$.

  To find $K_3$, we apply \Cref{prop:quasiconvex_subgroups_nest}
  twice: for $i = 1,2$, we take $G = H = G_i$, $J_1 = P$, $J_2 = J$,
  $U_1 = M \setminus \{p\}$, and $U_2 = M \setminus B_i$. As in the
  previous case we need to verify that
  $gB_i \subset M \setminus \{p\}$ for every $g \in G_i \setminus J$,
  but this follows because $gB_i \subset \Int(B_{3-i})$, which is
  disjoint from $\Lambda(J)$ and hence does not contain $p$.

  This gives us a pair of compact set $K_{3,1}$ and $K_{3,2}$, such
  that for any $g \in G_i \setminus J$, we can find $h \in P$ so that
  $hgB_i \subset K_{3,i}$. Then, since any $y \in T_1$ lies in
  $(G_1 \setminus J)(B_1) \cup (G_2 \setminus J)(B_2)$ by definition,
  we can take $K_3 = K_{3,1} \cup K_{3,2}$ and we are done.
\end{proof}

Finally, we can complete the proof of this direction of Theorem A part
\ref{item:AFP_gf}.

\begin{proposition}
  If $G_1$ and $G_2$ are geometrically finite, then $G$ is
  geometrically finite.
\end{proposition}
\begin{proof}
    Let $x \in \Lambda(G)$. We must show $x$ is either a conical limit point or a bounded parabolic point for $G$. First, if $x$ is not a translate of a limit point of $G_1$ nor $G_2$, then $x$ is a conical limit point by \Cref{lem:AFP_conical}. So, assume $x \in G(\Lambda(G_1) \cup \Lambda(G_2))$. Acting by elements of $G$ preserves the properties we are trying to show, so in fact we may assume $x \in \Lambda(G_1) \cup \Lambda(G_2)$. If $x$ is a parabolic point of $G$, we are done by \Cref{lem:AFP_bounded_parabolic_G}. Otherwise, $x$ is necessarily a conical limit point for $G_1$ or $G_2$ since these are geometrically finite, and again we are done since $x$ will also be a conical limit point for $G$.
\end{proof}

\subsection{Geometrical finiteness of the factors}

The last thing to do in this section is prove the other direction of
Theorem A part \ref{item:AFP_gf}, and show that $G_1$ and $G_2$ are
geometrically finite if $G$ is geometrically finite. The first step is
the following lemma, which makes use of the contraction property
proved earlier in this section.

\begin{lemma}
  \label{lem:AFP_conical_limit_seq_bounded}
  Assume that $G$ is geometrically finite. Let $x \in \Lambda(G_i)$
  for $i \in \{1,2\}$, and suppose that $(h_k)$ is a conical limit
  sequence in $G$ for $x$. Then, after extracting a subsequence, we
  can find some $h \in G$ so that $h_k \in hG_i$ for every $k$.
\end{lemma}
\begin{proof}
  Without loss of generality take $x \in \Lambda(G_1)$. Let $(h_k)$ be
  a conical limit sequence for $x$. This means that there are distinct
  points $a, b \in M$ such that $h_kx \to a$ and $h_kz \to b$ for any
  $z \in M \setminus \{x\}$.

  If there is some $h \in G$ so that $h_k \in hJ$ for infinitely many
  $k$, then we are done. So we may assume that, after taking a
  subsequence, each $h_k$ is an $(i,j)$-form for $i,j$ fixed, and each
  $h_k$ represents a different left $J$-coset in $G$. There are two
  cases to consider: either every $h_k$ is an $(i,1)$-form or every
  $h_k$ is an $(i,2)$-form.

  First suppose that $h_k$ is an $(i,2)$-form. By
  \Cref{lem:AFP_contraction}, after extraction the sets $(h_kB_2)$
  converge to a singleton. Since $x \in \Lambda(G_1) \subset B_2$, and
  $h_kx \to a$, we must have $h_kB_2 \to \{a\}$. Since $B_2$ is an
  infinite set by \Cref{prop:interactive_pair_infinite}, there is some
  point $z \in B_2 \setminus \{x\}$, which must satisfy
  $h_kz \to a$. But this is impossible if $(h_k)$ is a conical limit
  sequence for $x$.

  We conclude that each $h_k$ must be $(i,1)$-form. If $h_k \in G_1$
  for infinitely many $k$ then we are done, so assume that this is not
  the case. Then after taking a subsequence we have $|h_k| > 1$ for
  every $k$.  We write $h_k$ as an $(i,j)$-form of length $n \ge 2$:
  \[
    h_k = g_{k,1} \cdots g_{k,n}.
  \]
  Note that although $n$ can depend on $k$, we omit this from the
  notation. Since $h_k$ is an $(i,1)$-form, we have $g_{k,n} \in G_1$,
  and since $\Lambda(G_1)$ is $G_1$-invariant, $g_{k,n}x$ lies in
  $\Lambda(G_1) \subset B_2$. Then
  \[
    (h_kg_{k,n}^{-1}) = (g_{k,1} \cdots g_{k,n-1})
  \]
  is a sequence of $(i,2)$-forms. If the elements in this sequence lie
  in infinitely many different left $J$-cosets in $G$, then we extract
  a subsequence and apply \Cref{lem:AFP_contraction} to see that
  $(h_kg_{k,n}^{-1}B_2)$ again converges to a singleton. This
  singleton contains the limit of
  $(h_kg_{k,n}^{-1}g_{k,n}x) = (h_kx)$, so it is again equal to
  $\{a\}$. But then for any $z \in B_1 \setminus \{x\}$, we have
  $g_{k,n}z \in B_2$ and thus $h_kz = h_kg_{k,n}^{-1}g_{k,n}z \to a$,
  again giving a contradiction. We conclude that a subsequence of
  $(h_kg_{k,n}^{-1})$ lies in a single coset $hJ$ for $h \in G$, hence
  $h_k \in hJg_{k,n} \subset hG_1$.
\end{proof}

\begin{proposition}
  If $G$ is geometrically
  finite, then $G_1$ and $G_2$ are geometrically finite.
\end{proposition}
\begin{proof}
  Let $x \in \Lambda(G_1)$. We will show that $x$ is either a conical
  limit point or a bounded parabolic point for $G_1$. Since $G$ is
  geometrically finite, we know $x$ is either a conical limit point
  for $G$ or a bounded parabolic point for $G$.

  If $x$ is a conical limit point for $G$, then it has a conical limit
  sequence $(h_k)$ in $G$, i.e. a sequence such that $(h_kx,h_kz)$ lies
  in a compact subset of $(M \times M) \setminus \Delta$ for
  any $z \ne x$ in $M$. By \Cref{lem:AFP_conical_limit_seq_bounded},
  we know that, up to subsequence, $h_k = hg_k$ for $g_k \in G_1$ and
  $h$ fixed. Then $(g_k)$ is a conical limit sequence for $x$ in $G_1$
  and we are done.

  Otherwise, suppose $x$ is a bounded parabolic point for $G$. Let $P$
  be the stabilizer of $x$ in $G$. As we observed in the proof of
  \Cref{lem:AFP_bounded_parabolic_G}, part \ref{item:AFP_parabolics}
  of \Cref{prop:AFP_basic_properties} implies that $P$ is a subgroup
  of $G_1$.

  Since $x$ is a bounded parabolic point for $G$, again applying local
  compactness of $\Lambda(G) \setminus \{x\}$, there is a compact
  $K \subset \Lambda(G) \setminus \{x\}$ so that
  $P(K) = \Lambda(G) \setminus \{x\}$. Let
  $K_1 = K \cap \Lambda(G_1)$. Since $\Lambda(G_1)$ is closed, $K_1$
  is compact, and since $\Lambda(G_1)$ is $G_1$-invariant (hence
  $P$-invariant), we have
  \[
    P(K_1) = P(K \cap \Lambda(G_1)) = P(K) \cap \Lambda(G_1) =
    \Lambda(G_1) \setminus \{x\}.
  \]
  Thus $x$ is bounded parabolic for $G_1$ and we are done.
\end{proof}

\section{Combinatorial group theory: HNN extensions}
\label{sec:combinatorial_group_theory_hnn}

In this section we establish notation and give some basic facts about
\textit{HNN extensions}, in preparation for our second combination
theorem. Our main reference is again \cite{Maskit88}. In this section,
$M$ is again an arbitrary compact metrizable space, but as in
\Cref{sec:combinatorial_group_theory_afp}, these results are purely set-theoretic. We further assume throughout
this section that $G_0, G_1$ are subgroups of $\Homeo(M)$, where
$G_1 = \la f \ra$ is infinite cyclic, and $J_1, J_{-1}$ are subgroups
of $G_0$ with $fJ_{-1}f^{-1} = J_1$. We let $G$ denote
$\la G_0, G_1 \ra$, the subgroup of $\Homeo(M)$ generated by $G_0$ and
$G_1$. Note that conjugation by $f$ induces an abstract isomorphism
$J_{-1} \to J_1$, which we denote $f_*$. The indices are chosen to
make notation more convenient later.

As was the case for amalgamated free products, we can define HNN
extensions using equivalence classes of \emph{normal forms}.
\begin{definition}
  \label{defn:hnn_normal_form}
  A word $g = f^{\al_1}g_1 \cdots f^{\al_n}g_n$ in $f$ and elements
  $g_k$ of $G_0$ is a \emph{normal form} if:
  \begin{enumerate}[label=(\arabic*)]
  \item Each $g_k \in G_0$ is nontrivial for $k < n$;
  \item Each $\al_k$ is an integer, with $\al_k \ne 0$ whenever
    $k > 1$;
  \item\label{item:hnn_normal_technical_1} If $\al_k < 0$ and
    $g_{k-1} \in J_{-1} \setminus \{1\}$, then $\al_{k-1} < 0$;
  \item\label{item:hnn_normal_technical_2} If $\al_k > 0$ and
    $g_{k-1} \in J_1 \setminus \{1\}$, then $\al_{k-1} > 0$.
  \end{enumerate}
\end{definition}
Two words $g = f^{\al_1}g_1 \cdots f^{\al_n}g_n$ and
$h = f^{\beta_1}h_1\cdots f^{\beta_n}h_n$ are \textit{equivalent} if
we can obtain $g$ from $h$ by inserting finitely many conjugates and
inverses of words of the form $fjf^{-1}(f_*(j))^{-1}$ for
$j \in J_{-1}$ (words of this form are the identity in $G$). Every word of the form $f^{\al_1}g_1 \cdots f^{\al_n}g_n$
is equivalent either to a normal form or to the identity, which means
that every word in $f$ and elements of $G_0$ is equivalent to either a
normal form or the identity.

The \emph{length} of a normal form
$g = f^{\al_1}g_1 \cdots f^{\al_n}g_n$ is defined to be
$|g| = \sum_{i=1}^n|\al_i|$. Note that, in contrast to normal forms
for amalgamated free products, the length of a normal form
$f^{\al_1}g_1 \cdots f^{\al_n}g_n$ is \emph{not} necessarily
$n$. Length-0 normal forms correspond by definition to elements of
$G_0$.

If a normal form $g$ has positive length, $i \in \{0, \pm 1\}$ and
$j \in \{\pm 1\}$, then we say $g = f^{\al_1}g_1 \cdots f^{\al_n}g_n$
is an $(i,j)$-form if $\al_1$ is positive (resp. negative, zero) and
$i = 1$ (resp. $-1$, $0$), and $\al_n$ is positive (resp. negative)
and $j = 1$ (resp. $-1$). Our notation differs slightly from Maskit's, which will make some of our later arguments less cumbersome.

We set
\[
  G_0 *_f = \{\id\} \cup \{\text{equivalence classes of normal
    forms}\}.
\]
This set forms a group, with operation given by concatenation followed
by reduction to a normal form. It is called the \textit{HNN extension
  of $G_0$ by $f$}. Note that it is \emph{not} in general true that
the formal inverse of a normal form $g$ is also a normal form (see
\Cref{lem:HNN_inverse} below), but it is a formal product of normal
forms, which tells us that $G_0 *_f$ contains inverses.

We again have a natural surjective homomorphism
\begin{align*}
  \varphi : G_0 *_f &\to G\\
  f^{\al_1}g_1\cdots f^{\al_n}g_n &\mapsto f^{\al_1} \circ g_1 \circ
                                    \cdots \circ f^{\al_n} \circ g_n. 
\end{align*}
The map $\varphi$ may or may not be an isomorphism. As was the case
for amalagamated free products, if $\varphi$ is an isomorphism, we
will abuse notation and say that $G = G_0 *_f$. In this situation, we
implicitly identify elements of $G$ with equivalence classes of normal
forms in $G_0 *_f$.

As in \Cref{sec:combinatorial_group_theory_afp}, we want a
``ping-pong'' condition ensuring that $\varphi$ actually is an
isomorphism.
\begin{definition}\label{def:interactive_triple}
  Let $U_1,U_{-1} \subset M$ be nonempty disjoint sets, with
  $A = M \setminus (U_1 \cup U_{-1})$ nonempty. We call
  $(A,U_1,U_{-1})$ an \textit{interactive triple} for $G_0$ and $G_1$
  if the following hold:
    \begin{enumerate}
    \item The pair $(U_1,U_{-1})$ is precisely invariant under
      $(J_1, J_{-1})$ in $G_0$.
        \item For $i \in \{\pm 1\}$, and for every $g \in G_0$, $gU_i \subset A \cup U_i$.
        \item We have $f(A \cup U_1) \subset U_1$ and
          $f^{-1}(A \cup U_{-1}) \subset U_{-1}$.
    \end{enumerate}
    We say an interactive triple is \textit{proper} if the set
    $A \setminus(G_0(U_1 \cup U_{-1}))$ is nonempty.
\end{definition}

Note that these conditions imply that in particular $gU_i \subset A$
for $g \in G_0 \setminus J_i$. Similarly to \Cref{sec:combinatorial_group_theory_afp}, we can
observe:
\begin{proposition}
  \label{prop:interactive_triple_infinite}
  If $(A, U_1, U_{-1})$ is a proper interactive triple for $G_1$ and
  $G_2$, then $A$, $U$, and $U_{-1}$ are all infinite sets.
\end{proposition}
\begin{proof}
  Since $J_1$ is a proper subgroup $G_0$, there is some element
  $g \in G_0 \setminus J_1$, and by precise invariance we have
  $gU_1 \subset A$. By properness of the triple, the inclusion is
  proper, which means that $fgU_1$ is a proper subset of $U_1$. We
  conclude that $U_1$ is infinite. Since $gU_1 \subset A$ and
  $f^{-1}gU_1 \subset U_{-1}$ the other two sets are infinite as well.
\end{proof}

We have a description of the way normal forms in $G_0 *_f$ act on
certain sets in the interactive triple, in analogy to the way normal
forms in an amalgamated free product act on sets in an interactive
pair.

\begin{lemma}[\cite{Maskit88} VII.D.11]\label{lem:interactive_triple_dynamics}
  Suppose there is an interactive triple $(A, U_1, U_{-1})$ for $G_0$
  and $G_1$, and set $A_0 = A \setminus G_0(U_1 \cup U_2)$. Let
  $g = f^{\al_1}g_1 \cdots f^{\al_n}g_n \in G_0 *_f$ be a normal form
  with $|g| > 0$. Then the following hold.
  \begin{enumerate}[label=\roman*)]
  \item if $g$ is an $(i,j)$-form for $i,j \in \{\pm 1\}$, then
    $\varphi(g)(A_0 \cup U_j) \subset U_i$.
  \item If $g$ is a $(0,j)$-form for $j \in \{\pm 1\}$, then there is
    $h \in G_0$ so $\varphi(g)(A_0 \cup U_j) \subset hU \subset A$,
    where $U = U_{-1}$ if $\al_2 < 0$ and $U = U_1$ if $\al_2 > 0$.
  \end{enumerate}
\end{lemma}

The combinatorics in this case are slightly more complicated than for
amalgamated free products, but the basic idea is the same. To
illustrate the idea, consider a $(1,1)$-form of length 2, for example
$g = fg_1fg_2$. Then $g_2(A_0 \cup U_1) \subset A \cup U_1$ by
definition (in fact $A_0$ is $G_0$-invariant by our conditions). Then
we have
\begin{align*}
    g(A_0 \cup U_1) &\subset fg_1f(A \cup U_1)\\
    &\subset fg_1(U_1)\\
    &\subset f(A \cup U_1)\\
    &\subset U_1.
\end{align*}
Conditions \ref{item:hnn_normal_technical_1} and
\ref{item:hnn_normal_technical_2} in \Cref{defn:hnn_normal_form}
ensure that when we iteratively apply a normal form to $A_0 \cup U_i$,
we always can say where each set is mapped to next. We have chosen our
notation so that if $g$ is an $(i,j)$-form with $i \neq 0$, then
$gU_j \subset U_i$. This is consistent with the convention for
amalgamated free products.

The proposition below gives the combinatorial condition we need to
ensure that $\varphi$ is actually an isomorphism.

\begin{proposition}[Ping-pong for HNN extensions; \cite{Maskit88}
  VII.D.12]\label{prop:interactive_triple}
  Suppose $(A, U_1, U_{-1})$ is a proper interactive triple for $G_0$
  and $G_1$. Then $G = G_0 *_f$.
\end{proposition}
\begin{proof}
  We just need to show that $\varphi : G_0 *_f \to G$ is
  injective. This map is already injective on $G_0$, so suppose
  $g \in G_0 *_f$ has $|g| > 0$. Then by
  \Cref{lem:interactive_triple_dynamics} we have $\varphi(g)x \neq x$
  for any $x \in A_0 = A \setminus G_0(U_1 \cup U_{-1})$, showing that
  $\varphi(g)$ is not the identity.
\end{proof}

\subsection{More combinatorics of normal forms}

Normal forms in an HNN extension are slightly more complicated than
normal forms for an amalgamated free product, so here we collect some
results which will later make working with these normal forms a little
easier.

\subsubsection{Formal inverses}
In several situations later in the paper, we will want to work with
formal inverses of $(i,j)$-forms. These inverses may not themselves be
normal forms, but it is still useful to work with them directly,
rather than with an equivalent normal form. To that end, we prove:
\begin{lemma}\label{lem:HNN_inverse}
  Let $g = f^{\al_1}g_1\cdots f^{\al_n}g_n$ be an $(i,j)$-form for
  $i,j \in \{\pm 1\}$ (so in particular, $g$ has positive length). Then
  the formal inverse
  \[
    g^{-1} = g_n^{-1}f^{-\al_n} \cdots g_1^{-1}f^{-\al_1}
  \]
  is a $(0,-i)$-form if $g_n \in G_0 \setminus (J_1 \cup J_{-1})$. The
  word $f^{-\al_n}\cdots g_1^{-1}f^{-\al_1}$ is a $(-j,-i)$-form,
  regardless of $g_n$.
\end{lemma}
\begin{proof}
  We set $\beta_0 = 0$ and $\beta_k = -\alpha_{n + 1 -k}$ for
  $1 \le k < n$, so that $g^{-1}$ is the word
  \[
    f^{\beta_0}g_n^{-1}f^{\beta_1}g_{n-1}^{-1} \cdots f^{\beta_{n-1}}.
  \]
  We need to verify that this word is a normal form. The only
  conditions in \Cref{defn:hnn_normal_form} which could possibly fail
  are the technical requirements \ref{item:hnn_normal_technical_1} and
  \ref{item:hnn_normal_technical_2}.

  For \ref{item:hnn_normal_technical_1}, we must show that, for
  $k \ge 0$, if $\beta_{k+1} < 0$ and
  $g_{n- k}^{-1} \in J_{-1} \setminus \{\id\}$, then $\beta_k <
  0$. Equivalently, we need to show that if $\alpha_{n+1-k} > 0$ and
  $g_{n - k} \in J_{-1}$, then $\alpha_{n - k} < 0$. When $k \ge 1$
  this follows from condition \ref{item:hnn_normal_technical_2} on our
  original normal form $g$, and when $k = 0$ the condition is vacuous
  because we assume $g_n \notin J_{-1}$. The argument for condition
  \ref{item:hnn_normal_technical_2} is nearly identical.

  The same reasoning implies that
  $f^{\beta_1}g_{n-1}^{-1} \cdots f^{\beta_{n-1}}$ is a normal form,
  with $\beta_1 = -\alpha_n$ and $\beta_{n-1} = -\alpha_1$.
\end{proof}

\subsubsection{Ping-pong for normal forms}

When we have an interactive triple $(A, U_1, U_{-1})$ for an HNN
extension $G$, \Cref{lem:interactive_triple_dynamics} above gives us a
way to locate sets of the form $g U_i$ when $g$ is a normal form in
$G$. However, the statement of the lemma is often a little unwieldy to
work with directly, so to simplify some arguments later on, we
introduce some additional terminology.
\begin{definition}
  Let $G = G_0 *_f$ be the HNN extension of $G_0$ along
  $J_1 = f^{-1}J_{-1}f$. We say
  that a normal form
  \[
    g = f^{\alpha_1}g_1 \cdots f^{\alpha_n}g_n
  \]
  is an \textit{HNN ping-pong form of type 1} (or just a \textit{type-1
    form}) if either $g_n \in G_0 \setminus J_1$, or $\alpha_n > 0$. 
    Similarly a normal form is an \textit{HNN ping-pong form of type
    $-$1} if either $\alpha_n < 0$ or $g_n \in G_0 \setminus J_{-1}$. 
\end{definition}
Note that if $|g| = 0$, then $g$ has type $i$ if and only if $g \in G_0 \setminus J_i$. An $(i,j)$-form is always type $j$, and it may or may not also be type
$-j$. If $(A, U_1, U_{-1})$ is an interactive triple for $G_0$,
$\la f \ra$, then a normal form $g$ has type $k$ when the dynamics of
the triple allow us to locate the set $gU_k$. That is, we have the
following immediate consequence of
\Cref{lem:interactive_triple_dynamics}:
\begin{lemma}
  \label{lem:HNN_pingpong_type_mapping}
  Let $(A, U_1, U_{-1})$ be an interactive triple for $G_0$ and
  $\la f \ra$. If $g$ is an $(i,j)$-form of type $k$, and $i \ne 0$,
  then $gU_k \subset U_i$.
\end{lemma}

Frequently we will want to apply inductive arguments to normal forms,
which means that we want some control over the ping-pong behavior of a
prefix of an $(i,j)$-form. The lemma below gives one way to do this. Here (and elsewhere), a ``prefix'' $h'$ of a normal form
$h$ is a normal form which appears as an initial subword of $h$. That
is, if $h$ is a normal form $f^{\alpha_1}g_1 \cdots f^{\alpha_n}g_n$,
then a prefix $h'$ is a normal form
$f^{\alpha_1}g_1 \cdots f^{\alpha_k}g_k$ for some $1 \le k \le n$.
\begin{lemma}\label{lem:HNN_prefixes}
  Let $(A, U_1, U_{-1})$ be an interactive triple for $G_0$ and
  $\la f \ra$, and let $g$ be a type-$i$ normal form of length
  $m \ge 1$. Then for some $j \in \{-1, 1\}$, there is a
  length-$(m-1)$ prefix $g'$ of $g$ and $g_0 \in G_0$ so that
  $g = g'f^jg_0$ and $f^jg_0U_i \subset U_j$. If $|g'| \ge 1$, then
  $g'$ is type $j$.
\end{lemma}
\begin{proof}
  When $m = 1$ we can just take $g' = \id$, so assume $m > 1$. We let
  $g = f^{\al_1}g_1 \cdots f^{\al_n}g_n$ be a type-$i$ normal
  form. Without loss of generality assume $\alpha_n > 0$, and consider
  the normal form
  \[
    g' = gg_n^{-1}f^{-1} = f^{\al_1}g_1 \cdots f^{\al_{n-1}}g_{n-1}f^{\al_n -
      1}.
  \]
  This is a normal form with positive length $m-1$. It is also type 1:
  if $\al_n > 1$ or $g_{n-1} \in G \setminus J_1$, then this follows
  directly from the definition; if $\al_n = 1$ and $g_{n-1} \in J_1$,
  then, since $f^{\al_1}g_1 \cdots f^{\al_{n-1}}g_{n-1}f^{\al_n}g_n$
  is a normal form, we must have $\al_{n-1} > 0$, which again means
  the above form has type 1.

  We need to verify that $fg_nU_i \subset U_1$, which will show the lemma holds with $g_0 = g_n$. If $i = 1$, then
  $fg_nU_i = fg_nU_1 \subset U_1$. On the other hand, if $i = -1$,
  then since $\alpha_n > 0$ and $g$ has type $i$, we must have
  $g_n \in G \setminus J_{-1}$. Then
  $fg_nU_i = fg_nU_{-1} \subset U_{1}$.
\end{proof}

We will also sometimes want to characterize elements in $G$ via their
action on sets in an interactive triple $(A, U_1, U_{-1})$. This can
again be expressed using the ping-pong type of normal forms for these
elements. The lemma below is a precise statement of this form, and
generalizes the fact that in $G_0$, $(U_1, U_{-1})$ is precisely
invariant under $(J_1, J_{-1})$:
\begin{lemma}
  \label{lem:HNN_improved_precise_invariance}
  Let $(A, U_1, U_{-1})$ be an interactive triple for $G_0$ and
  $\la f \ra$. Let $g$ be a ping-pong form of type $i$ and let $h$ be
  a ping-pong form of type $k$.

  Suppose that $|g| = |h|$. Then either $i = k$, $gU_i = hU_i$, and
  $g = hj$ for $j \in J_i$, or $gU_i \cap hU_k = \es$.
\end{lemma}
\begin{proof}
  We proceed by induction on the length $m$ of $g$ and $h$; the main
  idea is to use the previous lemma to find prefixes of $g$ and $h$
  where we can assume that the statement holds, and then apply precise
  invariance of $(U_1, U_{-1})$ under $(J_1, J_{-1})$ for the
  inductive step.

  First observe that, if $g, h$ are elements of $G_0$, and if
  $gU_i \cap hU_k \ne \es$, then the fact that $(U_1, U_{-1})$ is
  precisely invariant under $(J_1, J_{-1})$ implies $i = k$ and
  $g = hj$ for $j \in J_i$. Now, let $m \ge 1$, let $g$, $h$ be normal
  forms with $|g| = |h| = m$, and suppose that $g$ has type $i$, $h$
  has type $k$, and $gU_i \cap hU_k \ne \es$.

  By \Cref{lem:HNN_prefixes} we can find prefixes $g', h'$ of type
  $i'$, $k'$ respectively, with $|g'| = |h'| = m-1$ and
  $g = g'f^{i'}g_0$, $h = h'f^{k'}h_0$ for $g_0, h_0 \in G_0$
  satisfying $f^{i'}g_0U_i \subset U_{i'}$ and
  $f^{k'}h_0U_k \subset U_{k'}$. Then we know that both
  $gU_i \subset g'U_{i'}$ and $hU_k \subset h'U_{k'}$ hold, which
  means that $h'U_{k'} \cap g'U_{i'} \ne \es$. By induction (or by
  precise invariance if $n = 1$), we know that $i' = k'$ and
  $h' = g'j'$ for $j' \in J_{i'}$. Without loss of generality take
  $i' = k' = 1$, so $j' \in J_1$.

  Then since $gU_i = g'fg_0U_i$ has nonempty intersection with
  $hU_k = h'fh_0U_k = g'j'fh_0U_k$, the intersection
  $fg_0U_i \cap j'fh_0U_k$ is also nonempty. Since $f$ conjugates
  $J_1$ to $J_{-1}$, for some $j'' \in J_{-1}$ we have
  $j'fh_0 = fj''h_0$. Then $fg_0U_i \cap fj''h_0U_k$ is nonempty as
  well, hence $g_0U_i \cap j''h_0U_k \ne \es$. Then by precise
  invariance we know $i = k$ and $g_0 = j''h_0j$ for $j \in J_i$.

  Finally, we see that
  \[
    g = g'fg_0 = g'fj''h_0j = g'j'fh_0j = h'fh_0j = hj,
  \]
  and we are done.
\end{proof}

\section{Theorem B}
\label{sec:theorem_b}

In this section we prove Theorem B. The proof is very similar in
spirit and structure to the proof of Theorem A, but the details are
different. Where possible, we have tried to imitate the structure of
\Cref{sec:theorem_a}, and have indicated the analogies between the
proofs.

We start (as in \Cref{sec:theorem_a}) by setting up the general
ping-pong framework.

\begin{definition B}[HNN Ping-Pong Position]\label{def:HNN_position}
  Let $G_0$ be a discrete convergence group acting on a compact
  metrizable space $M$, and suppose that $J_1, J_{-1} < G_0$ are both
  geometrically finite. Let $G_1 = \la f \ra$ be an infinite discrete
  convergence group also acting on $M$, where $fJ_{-1}f^{-1} = J_1$ in
  $\Homeo(M)$. We will say $G_0$ is in \textit{HNN ping-pong position}
  (with respect to $f, J_1$ and $J_{-1}$) if there exists closed sets
  $B_1, B_{-1} \subset M$ with nonempty disjoint interiors satisfying
  the following:
    \begin{enumerate}[label=(\arabic*)]
        \item \label{item:HNN_pingpong} $(B_1, B_{-1})$ is precisely
          invariant under $(J_1,J_{-1})$ in $G_0$ (recall
          \Cref{def:precisely_invariant}). 
        \item \label{item:HNN_f} If $A = M \setminus (B_1 \cup
          B_{-1})$, then $f(A \cup B_1) = \Int(B_1)$. 
        \item \label{item:HNN_technical} For $i \in \{\pm 1\}$, we
          have $\Lambda(G_0) \cap B_i = \Lambda(J_i)$.
        \item \label{item:HNN_proper} The set
          $A_0 = M \setminus G_0(B_1 \cup B_{-1})$ is nonempty.
    \end{enumerate}
\end{definition B}

\begin{remark}\label{rmk:HNN_disjointness}
    Note that our precise invariance assumption forces $B_1 \cap B_{-1} = \es$.
\end{remark}
We restate Theorem B here for reference.

\theoremB*

\begin{remark}
  As was the case for amalgamated free products, when $M = \del \H^3_\R$,
  this theorem is not strong enough to recover Maskit's full result,
  since we ask for stronger hypotheses on our ping-pong
  configuration. Specifically, we do not allow $B_1$ and $B_{-1}$ to
  intersect, and consequently $f$ cannot be parabolic. This condition
  ensures that our subgroups are fully quasi-convex, and allows us to
  apply \Cref{prop:quasiconvex_subgroups_nest}.
\end{remark}

Before proving the first three parts of Theorem B, we give the
following slightly stronger version of
\Cref{lem:HNN_pingpong_type_mapping}, which will be useful throughout
this section.
\begin{lemma}
  \label{lem:HNN_pingpong_strong_nesting}
  Suppose that $g$ is an $(i,j)$-form of type $k$. If $i \ne 0$, then
  $gB_k \subsetneq \Int(B_i)$, and if $i = 0$, then
  $gB_k \subsetneq A$.
\end{lemma}
\begin{proof}
  First suppose that $i \ne 0$. For concreteness, assume $g$ is a
  $(1,j)$-form. We first suppose that $|g| = 1$, so that $g = fg_1$
  for $g_1 \in G_0$. If $k = 1$, then $g_1B_1$ is a subset of
  $M \setminus B_{-1} = \Int(A \cup B_1)$ by precise invariance. In
  fact it is a proper subset by properness of the interactive triple,
  so $fg_1B_1 \subsetneq \Int(B_1)$ by condition \ref{item:HNN_f} in
  Definition B. If $k = -1$, then since $g$ has type $k$, we must have
  $g_1 \in G_0 \setminus J_{-1}$ and
  $g_1B_{-1} \subset M \setminus B_{-1} = \Int(A \cup B_1)$. Again,
  the inclusion is proper by properness of the interactive triple, so
  again we have $fg_1B_{-1} \subsetneq \Int(B_1)$.

  When $|g| > 1$, we can apply \Cref{lem:HNN_prefixes} and induction:
  we write $g = g'f^jg_n$, where $g'$ is a type-$j$ normal form with
  length $|g| - 1$, and $f^jg_nB_k \subset B_j$. Via induction we know
  that $g'B_j \subsetneq \Int(B_1)$, which means
  $gB_k \subsetneq \Int(B_1)$.

  The case $i = 0$ follows from the first case and precise invariance
  of $B_i$ under $J_i$, since any $(0, j)$-form $g$ can be written
  $g = g_1g'$, where $g'$ is an $(i,j)$-form and
  $g_1 \in G_0 \setminus J_i$.
\end{proof}

We now prove the first three parts of Theorem B. Again, the arguments
are standard.

\begin{proof}[Proof of \ref{item:HNN_extension} -
  \ref{item:HNN_parabolics} in Theorem B]
  \ref{item:HNN_extension} Let $B_1$ and $B_{-1}$ be the sets given by
  our conditions, and set $A = M \setminus (B_1 \cup B_{-1})$. Note
  that condition \ref{item:HNN_f} of Definition B implies
  $f^{-1}(A \cup B_{-1}) = \Int(B_{-1})$. The result now follows from
  \Cref{prop:interactive_triple} since $(A, \Int(B_1), \Int(B_{-1}))$ form an
  interactive triple which is proper by condition
  \ref{item:HNN_proper} of Definition B.
  
  \ref{item:HNN_discrete} It suffices to show no sequence $(g_k)$ in
  $G$ can accumulate at the identity. If $|g_k| = 0$, then $g_k$ lies
  in the discrete group $G_0$, so assume $|g_k| \ge 1$ for all $k$. We
  can consider several cases. If a normal form for $g_k$ ends in a
  power of $f$, then $g_kA_0 \subset B_1 \cup B_{-1}$. Otherwise,
  $g_k$ is either a $(0,1)$ or a $(0,-1)$-form. In the former case,
  $g_kB_1 \subset A$, and in the latter case $g_kB_{-1} \subset A$. In
  each of these cases, $g_k$ takes a fixed set with nonempty interior
  into another fixed disjoint set, which means $g_k$ cannot accumulate
  on the identity.
    
  \ref{item:HNN_parabolics} Let $g = f^{\al_1}g_1\cdots f^{\al_n}g_n$
  be a normal form not conjugate into $G_0$. Conjugating and replacing
  $g$ with $g^{-1}$ if necessary, we can assume that $\al_1 > 0$ (so
  $g$ is a $(1, j)$-form) and that $|g|$ is minimal in its conjugacy
  class. Note that if $\al_n < 0$ and $g_n \in J_1$, then
  $f^{-1}gf = f^{\al_1-1}g_1\cdots f^{\al_n+1}f^{-1}g_n f$ has a
  strictly smaller length than $g$ since $f^{-1}g_nf \in J_{-1}$, so
  we know that either $\al_n > 0$ or $g_n \in G_0 \setminus J_1$. That
  is, $g$ is a $(1,j)$-form of type 1, so by
  \Cref{lem:HNN_pingpong_strong_nesting}, $gB_1$ is a proper subset of
  $\Int(B_1)$. Then the same argument as in Theorem A part
  \ref{item:AFP_parabolics} implies that $g$ has infinite order, and a
  fixed point in $\Int(B_1)$.

  On the other hand if $gB_1$ is a proper subset of $B_1$, then
  $g^{-1}(M \setminus gB_1)$ is a proper subset of $M \setminus B_1$,
  so the same argument again shows that $g^{-1}$ has a fixed point in
  the closure of $M \setminus B_1$. Thus $g$ has two distinct fixed
  points and is loxodromic.
\end{proof}

\subsection{Limit sets of HNN extensions}
The remainder of the section is meant to prove part \ref{item:HNN_gf}
of Theorem B, so for the rest of the paper we fix the space $M$ and
groups $G_0, J, \la f \ra, G$ in $\Homeo(M)$ satisfying the
conditions of Definition B. As for Theorem A, we start by establishing
some properties of the limit points of $G$ under these assumptions.

\begin{proposition}\label{prop:HNN_basic_properties}
    With the above conditions and notation, each of the following holds.
    \begin{enumerate}[label=(\roman*)]
        \item\label{item:HNN_prop_i} $B_1 \cap B_{-1} = \es$, and $f$ is loxodromic with attracting fixed point in $\Int(B_1)$ and repelling fixed point in $\Int(B_{-1})$.
        \item\label{item:HNN_prop_ii} $\Lambda(J_{\pm 1}) \subset \del B_{\pm 1}$.
        \item\label{item:HNN_prop_iii} $\Lambda(G_0) \setminus G_0(\Lambda(J_1) \cup \Lambda(J_{-1})) \subset A_0$.
    \end{enumerate}
\end{proposition}
\begin{proof}
  \ref{item:HNN_prop_i} The fact that $B_1 \cap B_{-1} = \es$ follows
  from precise invariance of $(B_1, B_{-1})$ under $(J_1, J_{-1})$ in
  $G_0$. Now, since $f(A \cup B_1) = \Int(B_1)$, we have
  $f(\del B_{-1}) = \del B_1$, and so $fB_1 \subset
  \Int(B_1)$. Arguing as in the proof of Theorem A part
  \ref{item:AFP_parabolics}, we know this implies $f$ has a fixed
  point in $\Int(B_1)$ which is necessarily attracting. The same
  argument applied to $f^{-1}$ gives a fixed point in $\Int(B_{-1})$
  which is necessarily a repelling fixed point for $f$.

  \ref{item:HNN_prop_ii} Since $B_i$ is closed and $J_1$-invariant,
  and $\Int(B_i)$ is infinite by
  \Cref{prop:interactive_triple_infinite}, it follows that
  $\Lambda(J_i) \subset B_i$. Further, since $f$ conjugates $J_{-1}$
  to $J_1$, $f$ maps $\Lambda(J_{-1})$ bijectively onto
  $\Lambda(J_1)$. Hence
  \[
    \Lambda(J_1) = f\Lambda(J_{-1}) \subset fB_{-1} = B_{-1} \cup A
    \cup \del B_1.
  \]
  So we conclude
  $\Lambda(J_1) \subset (B_{-1} \cup A \cup \del B_1) \cap B_1 = \del
  B_1$ as desired. Applying an identical argument using $f^{-1}$ gives
  $\Lambda(J_{-1}) \subset \del B_{-1}$.
    
  \ref{item:HNN_prop_iii} Fix $x \in \Lambda(G_0)$, and suppose that
  $x \notin A_0$, i.e. that $x = gy$ for $g \in G_0$ and $y \in
  B_i$. Then $y = g^{-1}x \in \Lambda(G_0) \cap B_i$, which means
  $y \in \Lambda(J_i)$ by condition \ref{item:HNN_technical} in
  Definition B, and therefore $x \in G_0(\Lambda(J_i))$. It follows
  that $\Lambda(G_0) \setminus A_0$ is contained in
  $G_0(\Lambda(J_1) \cup \Lambda(J_{-1}))$, which is equivalent to the
  desired claim.
\end{proof}

\subsection{HNN ping-pong and contraction}

Next, we will establish a contraction lemma for HNN ping-pong
sequences, similar to \Cref{lem:AFP_contraction} for amalgamated free
products. As in the earlier case, the key tool is
\Cref{prop:quasiconvex_subgroups_nest}, so we start by establishing
that the subgroups $J_1$ and $J_{-1}$ are fully quasi-convex in some
ambient geometrically finite group.

First, we show:
\begin{lemma}
  \label{lem:HNN_boundary_invariance}
  Fix $i,j \in \{\pm 1\}$ and $g \in G$. Then
  $g \del B_i \cap \del B_j \ne \es$ if and only if either:
  \begin{enumerate}
  \item $i = j$ and $g \in J_i$, or
  \item $i = -j$ and $g = f^jh$ for $h \in J_i$.
  \end{enumerate}
\end{lemma}
\begin{proof}
  We induct on the length of $g$. If $|g| = 0$, then the claim follows
  from precise invariance of $(B_1, B_{-1})$ under $(J_1, J_{-1})$ in
  $G_0$, so suppose $|g| \ge 1$, and for concreteness, assume $i =
  1$. We write a normal form $f^{\al_1}g_1 \cdots f^{\al_n}g_n$ for
  $g$. If this normal form has type $1$, then
  \Cref{lem:HNN_pingpong_strong_nesting} implies that
  $g B_1 \subset \Int(B_1) \cup \Int(B_{-1}) \cup A$, hence
  $g \del B_1 \cap \del B_j = \es$. So we can assume that this normal
  form does \emph{not} have type $1$, which means that $\al_n < 0$ and
  $g_n \in J_1$. In this case, $f^{-1}g_n \del B_1 = \del B_{-1}$.

  The group element $g' = gg_n^{-1}f$ has strictly smaller length than
  $g$, so if $g'\del B_{-1} \cap \del B_j \ne \es$ then by induction
  we know that either $j = -1$ and $g' \in J_{-1}$, or $j = 1$ and
  $g' = fh$ for $h \in J_{-1}$. In the former case we can rewrite
  $g = g'f^{-1}g_n = f^{-1}g''g_n$ for $g'' \in J_1$, and in the
  latter case we can rewrite $g = g'f^{-1}g_n = fhf^{-1}g_n = g''g_n$
  for $g'' \in J_1$. Since $g_n \in J_1$ the conclusion follows.
\end{proof}

The lemma above implies in particular that $\del B_i$ is precisely
invariant under $J_i$ in both $G$ and $G_0$. Then, after applying part
\ref{item:HNN_prop_ii} of \Cref{prop:HNN_basic_properties}, we see:
\begin{corollary}
  Let $H$ be one of $G$ or $G_0$. If $H$ is geometrically finite, then
  $J_1$ and $J_{-1}$ are fully quasi-convex subgroups of $H$.
\end{corollary}

Now, we can apply \Cref{prop:quasiconvex_subgroups_nest} to the
present setting:
\begin{lemma}\label{lem:HNN_technical}
  Suppose that either $G$ or $G_0$ is geometrically finite. For
  $i \in \{\pm 1\}$, we can find a compact $K \subset A \cup B_{-i}$
  so that both of the following hold:

    \begin{enumerate}[label=(\roman*)]
    \item For any $g \in G_0 \setminus J_i$, we have $j \in J_i$ so
      $jgB_i \subset K$.
    \item For any $g \in G_0$, we have $j \in J_i$ so that
      $jgB_{-i} \subset K$.
    \end{enumerate}
\end{lemma}
\begin{proof}
  Take $i = 1$ to simplify notation. We can find a compact for each
  claim separately and take their union. First, we focus on (i). We
  can assume $B_1 \subset K$, so the statement follows immediately for
  $g \in J_1$ by taking $j$ to be the identity. Otherwise, we apply
  \Cref{prop:quasiconvex_subgroups_nest} with the ambient
  geometrically finite group as $G$ or $G_0$ (depending on which one
  is geometrically finite) and $H = G_0$ in both cases, and our two
  fully quasi-convex subgroups $J_1$ and $J_{-1}$ with corresponding
  invariant open sets $U_1 = M \setminus B_1$ and
  $U_{-1} = M \setminus B_{-1}$. Then if $g \in G_0 \setminus J_1$, we
  have $g(M \setminus U_1) = g(B_1) \subset A \subset U_{-1}$, and so
  the proposition gives our desired compact subset
  $K \subset U_{-1} = A \cup B_1$.

  For (ii), the proof is identical with $J_1$ playing the role of both
  fully quasi-convex subgroups in the statement of
  \Cref{prop:quasiconvex_subgroups_nest}, and both open sets being
  $M \setminus B_1 = A \cup B_{-1}$.
\end{proof}

We can now establish the HNN contraction property:
\begin{lemma}[Contraction for HNN
  extensions]\label{lem:HNN_contracting}
  Suppose that either $G$ or $G_0$ is geometrically finite, and let
  $(h_k)$ be a sequence of type-$i$ forms such that the left cosets
  $h_kJ_i$ are all distinct. Then up to subsequence, $(h_kB_i)$
  converges to a singleton $\{x\}$.
\end{lemma}
It follows from \Cref{lem:HNN_boundary_invariance} that any group
element $g \in G$ satisfying $gB_i = B_i$ must lie in $J_i$. So,
asking for the left cosets $(h_kJ_i)$ to be distinct is the same as
asking for the translates $(h_kB_i)$ to be distinct.
\begin{proof}
  To simplify notation, assume $i = 1$. The proof is very similar to
  the proof of \Cref{lem:AFP_contraction}. The first step is to show
  the following:
  \begin{claim*}
    After extracting a subsequence, there is a fixed $\ell = \pm 1$, a
    compact set $K \subset A \cup B_{-1}$, and a sequence $(j_k)$ in
    $J_1$ so that $j_kh_k^{-1}B_\ell \subset K$ for all $k$.
  \end{claim*}
  To prove the claim, we first suppose that $|h_k| = 0$. Then, since
  $h_k$ has type 1, we know $h_k \in G_0 \setminus J_1$. Then we take
  $\ell = 1$, and apply \Cref{lem:HNN_technical} to find the required
  set $K$ and elements $j_k$. Otherwise, suppose that $|h_k| \ge 1$,
  and write out a normal form for $h_k$:
  \[
    h_k = f^{\alpha_{k,1}}g_{k,1} \cdots f^{\alpha_{k,n}}g_{k,n}.
  \]
  We consider the inverse word
  \[
    g_{k,n}^{-1}f^{-\alpha_{k,n}} \cdots
    g_{k,1}^{-1}f^{-\alpha_{k,1}}.
  \]
  By \Cref{lem:HNN_inverse}, the sub-word
  $g_{k,n}h_k^{-1} = f^{-\alpha_{k,n}} \cdots
  g_{k,1}^{-1}f^{-\alpha_{k,1}}$ is a normal form, which must have
  length at least 1. Up to subsequence, for every $k$ this normal form
  is type $\ell$ for some fixed $\ell = \pm 1$, meaning that
  $g_{k,n}h_k^{-1}B_\ell \subset B_1$ if $-\alpha_{k,n} > 0$ and
  $g_{k,n}h_k^{-1}B_\ell \subset B_{-1}$ if $-\alpha_{k,n} < 0$. After
  extracting another subsequence we can assume one of these conditions
  holds for every $k$.

  In the case where $-\alpha_{k,n} < 0$ for every $k$, we can use
  \Cref{lem:HNN_technical} to find elements $j_k \in J_1$ and a
  compact $K \subset A \cup B_{-1}$ so that
  $j_kg_{k,n}^{-1}B_{-1} \subset K$ for every $k$. Then, we know that
  for every $k$, we have
  \[
    j_kh_k^{-1}B_\ell = j_kg_{k,n}^{-1}g_{k,n}h_k^{-1}B_\ell \subset
    j_kg_{k,n}^{-1}B_{-1} \subset K.
  \]
  On the other hand, if $-\alpha_{k,n} > 0$, then since $h_k$ has type
  1 we know that $g_{k,n} \in G_0 \setminus J_1$. Then again by
  \Cref{lem:HNN_technical} we can find a compact
  $K \subset A \cup B_{-1}$ and $j_k \in J_1$ so that
  $j_kg_{k,n}^{-1}B_1 \subset K$. Thus, we have
  \[
    j_kh_k^{-1}B_\ell = j_kg_{k,n}^{-1}g_{k,n}h_k^{-1}B_\ell \subset
    j_kg_{k,n}^{-1}B_1 \subset K.
  \]
  We have shown the claim above, so now consider the sequence
  $(h_kj_k^{-1})$. Since all the translates $h_kB_1$ are distinct, the
  group elements $h_k$ lie in infinitely many left $J_1$-cosets, hence
  so do the group elements $h_kj_k^{-1}$. In particular, the sequence
  $h_kj_k^{-1}$ is divergent in $G$, so we can extract a convergence
  subsequence and assume that there are points $z_+, z_- \in M$ so
  that $(h_kj_k^{-1}y)$ converges to $z_+$ whenever $y \ne
  z_-$. Equivalently, $(j_kh_k^{-1}y)$ converges to $z_-$ whenever
  $y \ne z_+$.

  \Cref{prop:interactive_triple_infinite} tells us that the set $B_1$
  is infinite, so in particular there must be some
  $y \in B_1 \setminus \{z_+\}$. Then, since
  $j_kh_k^{-1}B_1 \subset K$ we conclude that $z_- \in K$. Finally,
  since $B_1$ is a compact set in the complement of $K$, we see that
  $(h_kj_k^{-1}B_1) = (h_kB_1)$ must converge to $\{z_+\}$.
\end{proof}

\subsection{Geometrical finiteness of the extension}
We now prove that ($G_0$ geometrically finite) $\implies$ ($G$ geometrically finite). This gives one of the directions
of Theorem B part \ref{item:HNN_gf}.

As in the proof of the analogous direction of Theorem A, the key for
this direction of theorem is to show that limit points in
$\Lambda(G) \setminus \Lambda(G_0)$ can be ``coded'' by sequences of
$(i,j)$-forms in $G$. The precise statement is:
\begin{proposition}[HNN coding for $G$-limit
  points]\label{prop:HNN_limit_points_have_codings}
  Suppose that either $G$ or $G_0$ is geometrically finite, and let
  $x \in B_1 \cup B_{-1}$ be a point in
  $\Lambda(G) \setminus G(\Lambda(G_0))$. Then for fixed $\ell$, there
  is a sequence of type-$\ell$ forms $(h_k)$ in $G$ so that
  $|h_k| \to \infty$, each $h_k$ is a prefix of $h_{k+1}$, and
  $x \in h_kB_\ell$ for every $k$.
\end{proposition}
We can think of this proposition as a less explicit version of
\Cref{prop:AFP_limit_points_have_codings} in the amalgamated free
product case. The construction in this case is slightly more involved, and we need a little more
information about the location of certain points in $\Lambda(G)$. So,
we start by showing the following:
\begin{lemma}\label{lem:HNN_boundary_limit_points}
  Suppose that either $G$ or $G_0$ is geometrically finite. Then the
  only limit points of $G$ in $\del B_{\pm 1}$ are limit points of
  $J_{\pm 1}$. That is,
  $\Lambda(G) \cap \del B_{\pm 1} = \Lambda(J_{\pm 1})$.
\end{lemma}
\begin{proof}
  We will show that the intersection
  $\Lambda(G) \cap (\del B_1 \cup \del B_{-1})$ is a subset of
  $\Lambda(G_0)$; then we will be done by condition
  \ref{item:HNN_technical} in Definition B. So, let
  $x \in \Lambda(G) \cap \del B_1$. We can find a sequence $(g_k)$ in
  $G$ so that $g_kz \to x$ for all but perhaps a single $z \in M$. Now, if
  $g_k \in G_0$ for infinitely many $k$, the conclusion immediately
  follows. So we may assume $|g_k| \geq 1$ for every $k$.

  Up to subsequence, the $g_k$ are all $(i,j)$-forms for fixed
  $i \in \{0,\pm 1\}$ and $j \in \{\pm 1\}$. If $i = -1$, then $g_kB_j \subset \Int(B_{-1})$. But for some $z \in B_j$, the sequence $(g_kz)$ converges to $x \in \del B_1$. So, we know that either $i = 0$ or $i = 1$.

  If $i = 0$, then by \Cref{lem:interactive_triple_dynamics}, for some
  $\ell = \pm 1$ and some $h_k \in G_0 \setminus J_i$, we have
  $g_kB_j \subset h_kB_\ell \subset A$. There must be infinitely many
  distinct translates $h_kB_\ell$, since otherwise each $g_kz$ would
  lie in a fixed compact subset of $A$, and $(g_kz)$ could not
  converge to $x \in B_1$. So, the left cosets $h_kJ_\ell$ are all
  distinct. Then by \Cref{lem:HNN_contracting}, the sequence of sets
  $(h_kB_\ell)$ converges to a singleton, which must be $x$. But since
  $h_k \in G_0$ this again implies that $x \in \Lambda(G_0)$.

  Finally, we consider the case $i = 1$. We write out a normal form
  for $g_k$:
  \[
    g_k = f^{\al_{k,1}}g_{k,1} \cdots f^{\al_{k,n}}g_{k,n}.
  \]
  First observe that if $\al_{k,1} > 1$ for infinitely many $k$, then
  the word
  \[
    g_k' = f^{-1}g_k = f^{\al_{k,1} - 1}g_{k,1} \cdots
    f^{\al_{k,n}}g_{k,n}
  \]
  is still a $(1,j)$-form, which means that $f^{-1}g_kB_j \subset B_1$
  for infinitely many $k$. But then for infinitely many $k$, the point
  $g_kz$ lies in the compact subset $fB_1 \subset \Int(B_1)$, which is
  impossible if $g_kz \to x \in \del B_1$.

  We conclude that after extraction, we have $\al_{k,1} = 1$ for every
  $k$. After further extraction, we can assume that one of the the
  three conditions below holds for every $k$:
  \begin{enumerate}[label=(\alph*)]
  \item\label{item:length_zero} The length of $g_k'$ is zero;
  \item\label{item:0j_form} $g_k'$ is a $(0,j)$-form;
  \item\label{item:not_0j_form} $g_k'$ is not a normal form, hence $g_{k,1} \in J_1$ and $\al_{k,2} > 0$.
  \end{enumerate}
  If either \ref{item:length_zero} or \ref{item:0j_form} holds, we can
  use the first two cases of this proof to see that that $f^{-1}x$
  lies in $\Lambda(G_0) \cap \del B_{-1} = \Lambda(J_{-1})$, and thus
  $x \in \Lambda(J_1)$. And, \ref{item:not_0j_form} cannot occur: if
  $\al_{k,2} > 0$, then the word
  $f^{\al_{k,2}}g_{k,2} \cdots f^{\al_{k,n}}g_{k,n}$ is a
  $(1,j)$-form, which means $g_k'z \in g_{k,1}B_1$, and if
  $g_{k,1} \in J_1$ then $g_kz = fg_k'z \in fB_1$ and again $(g_kz)$
  cannot converge to $x \in \del B_1$.
\end{proof}

We now set about proving \Cref{prop:HNN_limit_points_have_codings}. As
in the analogous situation in the amalgamated free product case, we
follow Maskit's strategy by defining certain ``ping-pong'' sets in
$M$. Let $T_{0,i} = G_0(B_i)$, and
$T_0 = T_{0,1} \cup T_{0,-1} = G_0(B_1 \cup B_{-1})$, the union of all
$G_0$ translates of $B_1$ and $B_{-1}$.

More generally, let
\[
  T_{m,-1} = \bigcup gB_{-1},
\]
where the union is taken over length-$m$ normal forms $g$ of type
$-1$.  Similarly, let
\[
    T_{m,1} = \bigcup gB_1,
\]
where the union is taken over length-$m$ normal forms of type $1$. Let
\[
  T_m = T_{m,1} \cup T_{m,-1}.
\]

\Cref{lem:HNN_prefixes} implies that the sets $T_m$ are decreasing:
for any length-$m$ normal form $g$ with type $i$, we can use the lemma
to find a length-$(m-1)$ normal form $g'$ with type $j$, and
$g_0 \in G_0$ so that $gB_i = g'f^jg_0B_i \subset g'B_j$. So, we can
now consider the set
\[
  T = \bigcap_{m=0}^\infty T_m.
\]

\begin{figure*}[ht!]
    \centering
    \def\svgwidth{14.7 cm}
    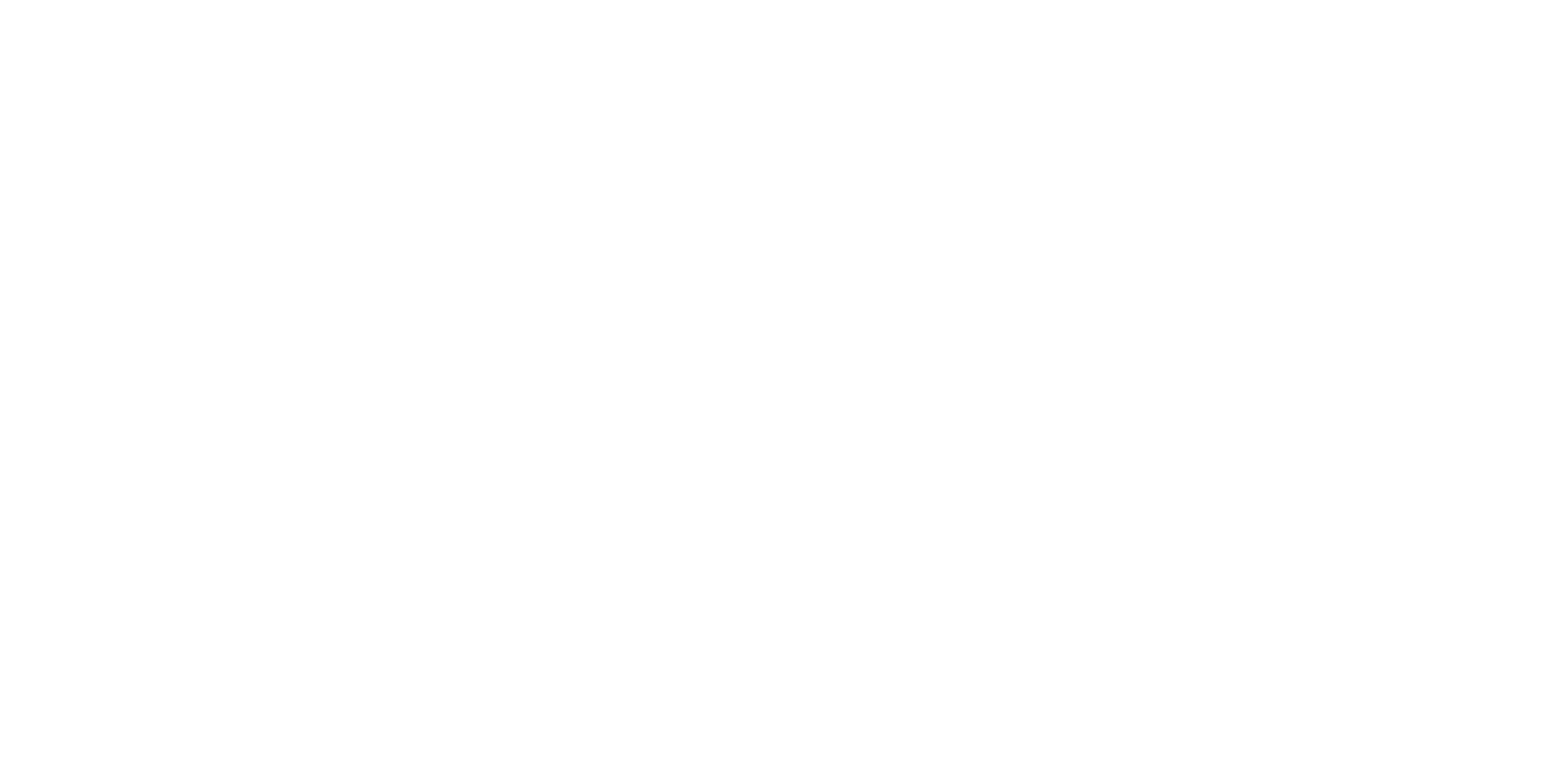
    \caption{Part of the sets $T_0$ and $T_1$.}
    \label{fig:HNN_T}
\end{figure*}

The proof of \Cref{prop:HNN_limit_points_have_codings} mainly involves
showing that $\Lambda(G) \setminus G(\Lambda(G_0)) \subset T$. Then,
we construct the desired sequence using the definition for $T$. The
first step is:
\begin{lemma}\label{lem:HNN_first_pingpong_inclusion}
  Suppose either $G$ or $G_0$ is geometrically finite. Then
  $\Lambda(G) \setminus \Lambda(G_0) \subset T_0$.
\end{lemma}
\begin{proof}
  This argument is similar to the proof of
  \Cref{lem:HNN_boundary_limit_points}. Suppose $y \in \Lambda(G)$
  does not lie in $T_0$ (that is, $y \in A_0$). We must show
  $y \in \Lambda(G_0)$. Since $y \in \Lambda(G)$, we can find a
  sequence $(g_k)$ in $G$ so that $g_kw \to y$ for all but a single
  $w \in M$. If $g_k \in G_0$ for infinitely many $k$ we will have
  $y \in \Lambda(G_0)$ as desired, so now suppose that $|g_k| \ge 1$
  for infinitely many $k$. After extracting a subsequence we can
  assume that each $g_k$ is an $(i,j)$-form for $i,j$ fixed. Since
  $B_j$ is an infinite set, we can fix some $w \in B_j$ so that
  $(g_kw)$ converges to $y$.

  By definition, we know that $g_kw \in T_n$, so in particular
  $g_kw \in T_0$ for every $k$. Then, we can extract a further
  subsequence so that $g_kw \in G_0(B_i)$ for fixed $i$ and write
  $g_kw = g_k'z_k$ for $g_k' \in G_0$ and $z_k \in B_i$.

  As $y \notin T_0$, there must be infinitely many distinct translates
  $g_k'B_i$, because otherwise every $g_k'z_k$ would lie in a fixed
  compact subset of $T_0$. Thus there are infinitely many distinct
  cosets $g_k'J_i$, and \Cref{lem:HNN_contracting} tells us that after
  extraction, $(g_k'B_i)$ must converge to a singleton. Since
  $g_k'z_k \to y$, it follows that this singleton is $y$. Hence for
  any choice of $z \in B_i$ we have $g_k'z \to y$, so
  $y \in \Lambda(G_0)$.
\end{proof}

\begin{proof}[Proof of \Cref{prop:HNN_limit_points_have_codings}]
  We first prove that
  $\Lambda(G) \setminus G(\Lambda(G_0)) \subset T$. So, fix
  $z \in \Lambda(G)$, and suppose $z \notin T$. We will show
  $z \in G(\Lambda(G_0))$.

  If $z \in \Lambda(G_0)$ we are done, hence by
  \Cref{lem:HNN_first_pingpong_inclusion} we can assume $z \in
  T_0$. Then we can find $m > 0$ so that $z \in T_{m-1} \setminus T_m$
  since these sets are decreasing. Without loss of generality, we have
  $z \in gB_{-1}$ for $g = f^{\al_1}g_1 \cdots f^{\al_n}g_n$ a normal
  form with length $m-1$ and type $-1$. If $g^{-1}z \in \del B_{-1}$,
  then since $\Lambda(G)$ is $G$-invariant we have
  $g^{-1}z \in \Lambda(G) \cap \del B_1 = \Lambda(J_{-1})$ by
  \Cref{lem:HNN_boundary_limit_points} and we are done. So suppose
  $g^{-1}z \in \Int(B_{-1})$.
    
  Since $z \notin T_m$, we have $fg^{-1}z \notin B_{-1}$ since
  $gf^{-1}$ has length $m$. Also, $fg^{-1}z \notin B_1$ since $f$ does
  not map any points of $\Int(B_{-1})$ into $B_1$. It follows that
  $fg^{-1}z \notin B_1 \cup B_2$, but also, $fg^{-1}z$ cannot be in a
  translate of $B_1$ nor $B_2$. Indeed, if $fg^{-1}z = hy$ for
  $y \in B_i$ and $h \in G_0$, then $h \in G_0 \setminus J_i$ since
  $hy \not\in B_i$ and $B_i$ is $J_i$-invariant. Hence
  $z = gf^{-1}hy \in T_m$, a contradiction. Hence
  $fg^{-1}z \notin T_0$, and so by
  \Cref{lem:HNN_first_pingpong_inclusion} we have
  $fg^{-1}z \in \Lambda(G_0)$ and $z \in G(\Lambda(G_0))$ as desired.

  We have now shown that
  $\Lambda(G) \setminus G(\Lambda(G_0)) \subset T$, so consider
  $z \in T$. We will construct our sequence $(h_k)$ of normal forms
  inductively. We know that $z \in T_1$, so we can find some normal
  form $h_1$ with type $i_1$ so that $z \in h_1B_{i_1}$. Now, assume
  that we have constructed a normal form $h_k$ of type $i$ so that
  $z \in h_kB_i$. Since $z \in T_{k+1}$, we can find a normal form
  $h_{k+1}'$ with length $k+1$ and type $\ell$ so that
  $x \in h_{k+1}'B_{\ell}$. Then, by \Cref{lem:HNN_prefixes}, there is
  a type-$i'$ form $h_k'$ with length $k$ and $g_0 \in G_0$ so that
  $h_{k+1}' = h_k'f^{i'} g_0$ and
  $h_{k+1}'B_{\ell} \subset h_k'B_{i'}$. Then $h_k'B_{i'}$ has
  nonempty intersection with $h_kB_i$, so by
  \Cref{lem:HNN_improved_precise_invariance} we have $i = i'$ and
  $h_kj = h_k'$ for $j \in J_i$. We can write $jf^i = f^ij'$ for
  $j' \in J_{-i}$. Then since $h_k$ has type $i$, $h_k$ is a prefix of
  the type-$\ell$ form $h_{k+1} = h_kf^ij'g_0$. This form is
  equivalent in $G$ to the type-$\ell$ form $h_{k+1}'$, hence
  $z \in h_{k+1}B_\ell$.

  Finally, by taking a subsequence, we can assume that each $h_k$ is a
  form of type $\ell$ for $\ell$ fixed, and we are done.
\end{proof}

As for amalgamated free products, we can use the coding given by
\Cref{prop:HNN_limit_points_have_codings} to construct conical limit
sequences for points in $\Lambda(G) \setminus G(\Lambda(G_0))$:
\begin{lemma}\label{lem:HNN_conical}
  If $G_0$ is geometrically finite, then every point of
  $\Lambda(G) \setminus G(\Lambda(G_0))$ is a conical limit point for
  $G$.
\end{lemma}
\begin{proof}
  Let $x \in \Lambda(G) \setminus
  G(\Lambda(G_0))$. \Cref{prop:HNN_limit_points_have_codings} says
  that for $i$ fixed, we can find a sequence $(h_k)$ of ping-pong
  forms of type $i$, with $|h_k| \to \infty$, so that each $h_k$ is a
  prefix of $h_{k+1}$, and $x \in h_kB_i$ for all $k$. Possibly after
  relabeling we may assume $i = 1$.

  We write $h_k$ in a normal form:
  \[
    f^{\alpha_1}g_1 \cdots f^{\alpha_{n_k}}g_{n_k}.
  \]
  If $\alpha_1 = 0$, then $\alpha_2 \ne 0$, in which case
  $h_k' = f^{\alpha_2}g_2 \cdots f^{\alpha_{n_k}}g_{n_k}$ is a
  ping-pong form of type $1$ such that $h_k'B_1$ contains
  $x' = g_1^{-1}x$. Since $x'$ is a conical limit point
  if and only if $x$ is, if necessary we can replace $x$ with $x'$ and
  $h_k$ with $h_k'$, and assume that $\alpha_1 \ne 0$. That is, $h_k$
  is an $(\ell,j)$-form for $\ell \ne 0$, so $h_kB_1 \subset B_\ell$.

  Further, since $x \in h_1B_1$, by replacing $x$ with $h_1^{-1}x$ and
  $h_k$ with $h_1^{-1}h_k$, we can assume that also $x \in B_1$, hence
  $h_kB_1 \cap B_1 \ne \es$. Since $h_kB_1 \subset B_\ell$ we have
  $\ell = 1$, meaning $\alpha_1 > 0$.

  Now, consider the sequence of sets
  \[
    (h_k^{-1}B_{-1}) = (g_{n_k}^{-1}f^{-\alpha_{n_k}} \cdots
    g_1^{-1}f^{-\alpha_1}B_{-1}).
  \]
  By \Cref{lem:HNN_inverse}, the word
  $f^{-\alpha_{n_k}} \cdots g_1^{-1}f^{-\alpha_1}$ is a normal form;
  since $\alpha_1 > 0$ it is a form of type $-1$, implying that
  $f^{-\alpha_{n_k}} \cdots g_1^{-1}f^{-\alpha_1}B_{-1}$ is a subset
  of $B_1$ if $\alpha_{n_k} < 0$ and a subset of $B_{-1}$ if
  $\alpha_{n_k} > 0$. And, since $h_k$ is a form of type $1$, we know
  that either $g_{n_k} \in G_0 \setminus J_1$ or $\alpha_{n_k} > 0$.

  To prove that $x$ is conical, we want to apply
  \Cref{lem:conical_characterization}, which means we need to produce
  distinct elements $g_k$, a set $Y$ with at least two points, and
  disjoint compact sets $K_1$ and $K_2$ so that $g_kx \in K_2$ and
  $g_kY \subset K_1$. Let $K \subset A \cup B_{-1} = M \setminus B_1$
  be the compact from \Cref{lem:HNN_technical}, and take
  $Y = B_{-1}, K_1 = K$ and $K_2 = B_1$. We know $Y$ contains at least
  two points from \Cref{prop:interactive_triple_infinite}, so we just
  need to produce the sequence $(g_k)$ by modifying $h_k^{-1}$.

  For each fixed $k$, we already have $h_k^{-1}x \in B_1$ as
  desired. If $\al_{n_k} > 0$, then
  $h_k^{-1}B_{-1} \subset g_{n_k}^{-1}B_{-1}$. From the definition of
  $K$, we can find $j_k \in J_1$ so that
  $j_kg_{n_k}^{-1}B_{-1} \subset K$, hence
  $j_kh_k^{-1}B_{-1} \subset K$.

  On the other hand, if $\al_{n_k} < 0$, then we necessarily have
  $g_{n_k} \in G_0 \setminus J_1$, and
  $h_k^{-1}B_{-1} \subset g_{n_k}^{-1}B_1 \subset A$. Again using the
  definition of $K$, we can find $j_k \in J_1$ so that
  $j_kg_{n_k}^{-1}B_1 \subset K$, hence $j_kh_k^{-1}B_{-1} \subset K$.

  In either of these cases, we have $j_kh_k^{-1}B_{-1} \subset K$ and
  $j_kh_k^{-1}x \in j_kB_1 = B_1$, which means we can take
  $g_k = j_kh_k^{-1}$ to complete the proof.
\end{proof}

We next consider parabolic points.

\begin{lemma}\label{lem:HNN_bounded_parabolic_G}
  Suppose that $G_0$ is geometrically finite. If $p \in \Lambda(G_0)$
  is a parabolic point for the action of $G_0$ on $\Lambda(G_0)$, then
  $p$ is a bounded parabolic point for the action of $G$ on
  $\Lambda(G)$.
\end{lemma}
\begin{proof}
  Let $p \in \Lambda(G_0)$ be a parabolic point for $G$, and let $P$
  be the stabilizer of $p$ in $G$. Since $p$ is a bounded parabolic
  point, and $P$ contains the stabilizer of $p$ in $G_0$, we know that
  there is a compact $\widehat{K} \subset \Lambda(G_0) \setminus \{p\}$ so
  that $P(\widehat{K}) = \Lambda(G_0) \setminus \{p\}$. We want to find a
  compact $K \subset \Lambda(G) \setminus \{p\}$ so that
  $P(K) = \Lambda(G) \setminus \{p\}$.

  As in the proof of \Cref{lem:AFP_bounded_parabolic_G}, our strategy
  is to show that $\Lambda(G) \setminus \{p\}$ can be decomposed into
  several pieces, such that each piece is either far away from $p$ to
  begin with, or can be pushed uniformly far away from $p$ using
  either the boundedness of $p$ in $\Lambda(G_0)$ or an application of
  \Cref{prop:quasiconvex_subgroups_nest}. We consider two cases:
  \subsubsection*{Case 1:
    $p \in \Lambda(G_0) \setminus G_0(\Lambda(J_1) \cup
    \Lambda(J_{-1}))$}
  In this case, \Cref{lem:HNN_first_pingpong_inclusion} tells us that
  each point in $\Lambda(G) \setminus \{p\}$ lies in the union
  $\Lambda(G_0) \cup T_0$. We can further decompose $T_0$ by
  intersecting it with $(B_1 \cup B_{-1})$ and its complement $A$,
  meaning we decompose $\Lambda(G) \setminus \{p\}$ into three pieces
  lying in
  \[
    L_1 = \Lambda(G_0), \quad L_2 = (B_1 \cup B_{-1}), \quad L_3 = T_0
    \cap A.
  \]
  For each $L_i$, we need to find a compact set
  $K_i \subset M \setminus \{p\}$ so that if
  $y \in \Lambda(G) \cap L_i$, then we can find $h \in P$ so that
  $hy \in K_i$. Then we can take $K = K_1 \cup K_2 \cup K_3$.

  We know we can take $K_1 = \widehat{K}$ from the boundedness of $p$
  in $\Lambda(G_0)$, and from part \ref{item:HNN_technical} of
  Definition B we know that $B_1 \cup B_{-1}$ is already a compact
  subset of $M \setminus \{p\}$. So, we just need to find the compact
  set $K_3$.

  We apply \Cref{prop:quasiconvex_subgroups_nest}, taking $H = G_0$,
  $J_1 = P$, $U_1 = M \setminus \{p\}$, $J_2 = J_{\pm 1}$, and
  $U_2 = M \setminus B_{\pm 1}$, to see that there are sets
  $K_+, K_- \subset M \setminus \{p\}$ such that for any
  $g \in G_0 \setminus J_{\pm 1}$, we can find $h \in P$ so that
  $hgB_{\pm 1} \subset K_{\pm}$. To justify the application of the
  proposition, we need to check that for every
  $g \in G_0 \setminus J_{\pm 1}$, we have
  $gB_{\pm 1} \subset M \setminus \{p\}$, but this follows from part
  \ref{item:HNN_prop_iii} of \Cref{prop:HNN_basic_properties}. Then,
  we take $K_3 = K_+ \cup K_-$.

  Now, if $y \in T_0 \cap A$, then by definition we know that
  $y \in (G_0 \setminus J_1)(B_1) \cup (G_0 \setminus
  J_{-1})(B_{-1})$. But then by definition of $K_{\pm}$ we know that
  we can find $h \in P$ so that $hy \in K_+ \cup K_{-} = K_3$ and we
  are done.

  \subsubsection*{Case 2:
    $p \in G_0(\Lambda(J_1) \cup \Lambda(J_{-1}))$} Since $G$ acts by
  homeomorphisms it suffices to consider
  $p \in \Lambda(J_1) \cup \Lambda(J_{-1})$. Without loss of
  generality take $p \in \Lambda(J_1)$. As in the previous case, we
  decompose $\Lambda(G) \setminus \{p\}$ into several different
  pieces, by writing $M$ as the union
  \[
    M = fB_1 \cup fA \cup \del B_1 \cup A \cup B_{-1}.
  \]
  Since $p \in \del B_1$, the sets $fB_1 \subset \Int(B_1)$ and
  $B_{-1}$ are compact sets in the complement of $p$.

  So, we only need to consider the three pieces of
  $\Lambda(G) \setminus \{p\}$ contained in the three sets
  \[
    \del B_1, \quad A, \quad fA.
  \]
  We can further decompose these pieces by intersecting each of them
  with the sets $\Lambda(G_0)$, $f\Lambda(G_0)$ and their complements
  in $M$. By \Cref{lem:HNN_boundary_limit_points}, we know that
  $\del B_1 \cap \Lambda(G) \subset \Lambda(G_0)$. Also, from
  \Cref{lem:HNN_first_pingpong_inclusion}, we know that
  $\Lambda(G) \setminus \Lambda(G_0)$ lies in $T_0$, which means we
  now only need to consider the pieces of $\Lambda(G) \setminus \{p\}$
  contained in the four sets
  \[
    L_1 = \Lambda(G_0), \quad L_2 = T_0 \cap A, \quad L_3 =
    f\Lambda(G_0), \quad L_4 = f(T_0 \cap A).
  \]
  We want to find compact sets
  $K_1, K_2, K_3, K_4 \subset M \setminus \{p\}$ so that for each
  $y \in L_i \cap (\Lambda(G) \setminus \{p\})$, we can find $h \in P$
  so that $hy \in K_i$.

  We already know that we can take $K_1 = \widehat{K}$, and to find $K_2$,
  we can use the exact same construction we used for $K_3$ in Case
  1. To justify the application of
  \Cref{prop:quasiconvex_subgroups_nest} in this situtation, we again
  need to check that for any $g \in G_0 \setminus J_{\pm 1}$, we have
  $gB_{\pm 1} \subset M \setminus \{p\}$. This time, the desired
  inclusion follows from precise invariance of $(B_1, B_{-1})$ under
  $(J_1, J_{-1})$ and the fact that $p \in B_1$.

  \begin{figure}[h]
    \centering
    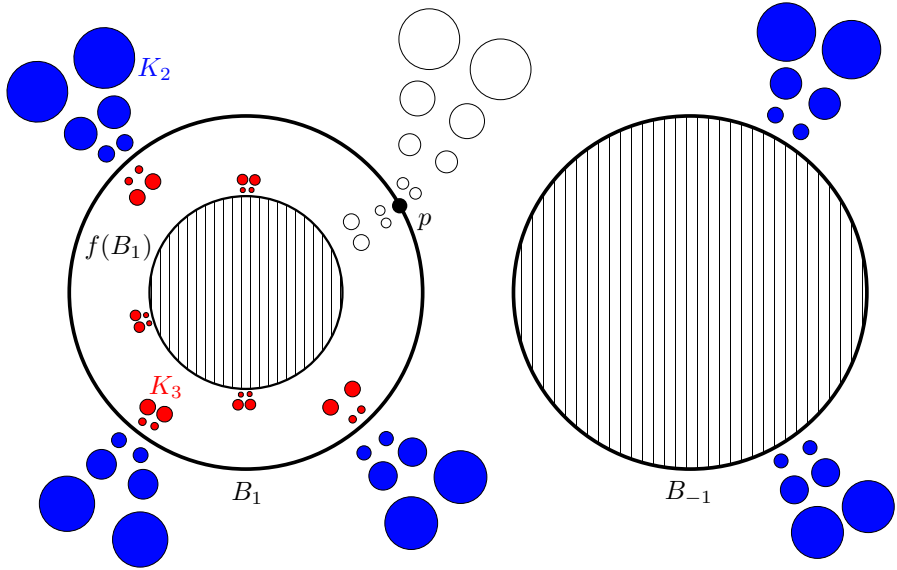
    \caption{Illustration for Case 2 of
      \Cref{lem:HNN_bounded_parabolic_G}. The sets $B_{-1}$ and
      $f(B_1)$ are already compact subsets of $M \setminus \{p\}$, so
      we need to divide the rest of $\Lambda(G)$ into pieces. The sets
      $K_1$ and $K_4$ (not pictured) lie in
      $\Lambda(G_0) \setminus \{p\}$ and
      $f(\Lambda(G_0)) \setminus \{p\}$.}
    \label{fig:hnn_bounded_parabolic_2}
  \end{figure}

  Finally, to find $K_3$ and $K_4$, we just apply the same exact
  arguments to the parabolic point $f^{-1}p \in \Lambda(J_{-1})$ and
  its stabilizer $f^{-1}Pf$, to obtain a pair of compact sets
  $K_3', K_4' \subset M \setminus \{f^{-1}p\}$ such that for any
  $z \in (\Lambda(G) \setminus \{f^{-1}p\}) \cap (L_1 \cup L_2)$, we
  can find $h \in P$ so that $f^{-1}hfz \in K_3' \cup K_4'$. We can
  take $K_3 = fK_3'$ and $K_4 = fK_4'$ (see
  \Cref{fig:hnn_bounded_parabolic_2}). Then if
  $y \in (\Lambda(G) \setminus \{p\}) \cap (L_3 \cup L_4) =
  (\Lambda(G) \setminus \{p\}) \cap (fL_1 \cup fL_2)$, we have
  $y = fz$ for
  $z \in (\Lambda(G) \setminus \{f^{-1}p\}) \cap (L_1 \cup L_2)$, and
  we can find $h \in P$ so that
  $f^{-1}hfz \in f^{-1}K_3 \cup f^{-1}K_4$, hence
  $hy \in K_3 \cup K_4$.
\end{proof}

Finally, we complete the proof of this direction of Theorem B part
\ref{item:HNN_gf}:
\begin{proposition}
  If $G_0$ is geometrically finite, then $G$ is geometrically finite.
\end{proposition}
\begin{proof}
  We must show that any $x \in \Lambda(G)$ is a conical limit point or
  bounded parabolic point. By \Cref{lem:HNN_conical}, we may assume
  $x \in G(\Lambda(G_0))$. Since $G$ acts by homeomorphisms, in fact we
  can assume that $x \in \Lambda(G_0)$. Since $G_0$ is geometrically
  finite, $x$ is either a conical limit point or bounded parabolic
  point for the action of $G_0$ on $\Lambda(G_0)$. In the former case,
  $x$ is also a conical limit point for $G$ acting on $\Lambda(G)$,
  and in the latter case $x$ is a bounded parabolic point for $G$
  acting on $\Lambda(G)$ by \Cref{lem:HNN_bounded_parabolic_G}. Hence
  $G$ is geometrically finite.
\end{proof}

\subsection{Geometrical finiteness of $G_0$}
Finally, we prove the other direction of Theorem B part
\ref{item:HNN_gf}, and show that if $G$ is geometrically finite, then
so is $G_0$. As for the amalgamated free product case, the first step
is the following:

\begin{lemma}\label{lem:HNN_conical_limit_seq_bounded}
  Assume that $G$ is geometrically finite. Let
  $x \in \Lambda(G_0) \setminus G_0(\Lambda(J_1) \cup
  \Lambda(J_{-1}))$, and suppose that $h_k \in G$ is a conical limit
  sequence for $x$. Then, after extracting a subsequence, we can find
  some $h \in G$ so that $h_k \in hG_0$ for every $k$.
\end{lemma}
\begin{proof}
  By \Cref{prop:HNN_basic_properties} part \ref{item:HNN_prop_iii}, we
  know $x \in A_0 \subset A$. As $x$ is a conical limit point, we can
  find a conical limit sequence $(h_k)$ for $x$, so that for distinct
  points $a, b \in M$, we have $h_kx \to a$ and $h_kz \to b$ for any
  $z \in M \setminus \{x\}$.
    
  If $h_k \in G_0$ for infinitely many $k$ then we are done, so we may
  assume $|h_k| \geq 1$ for every $k$. Suppose we can write
  $h_k = h_k'fg_k$ where $|h_k'| = |h_k|-1$ (the case where
  $h_k = h_kf^{-1}g_k$ is similar). We note that
  $g_kx \in A_0 \subset A$ still since $A_0$ is $G_0$-invariant.
    
  Consider the sequence $(h_kg_k^{-1}) = (h_k'f)$. We know that
  $fg_kx \in fA \subset B_1$, so $h_kx$ lies in $h_k'B_1$ for every
  $k$. If the $h_k'$ are all in distinct left $J_1$-cosets in $G$, the
  sequence $(h_k'B_1)$ converges to a singleton by
  \Cref{lem:HNN_contracting}. The limit of $(h_kx)$ is contained in
  this singleton, so the singleton is $\{a\}$. On the other hand,
  since $A$ is infinite, the set $g_k^{-1}A$ is also infinite, so
  there is at least one point $z$ in $g_k^{-1}A \setminus \{x\}$. But
  then $h_kz \in h_kg_k^{-1}A$, so we have
  \[
    h_kz \in h_k'fg_kg_k^{-1}A \subset h_k'fA \subset h_k'B_1.
  \]
  This means that $(h_kz)$ converges to $a$, which contradicts the
  fact that $(h_k)$ is a conical limit sequence for $x$.

  So, after taking a subsequence, we must have $h_k' \in h'J_1$ for
  some fixed $h' \in G$. Then for every $k$, we have
  $h_k \in h'J_1fg_k = h'fJ_{-1}g_k \subset h'fG_0$, and we are done.
\end{proof}

\begin{proposition}
    If $G$ is geometrically finite, then $G_0$ is geometrically finite.
\end{proposition}
\begin{proof}
  We must show that any $x \in \Lambda(G_0)$ is a conical limit point
  or bounded parabolic point for the $G_0$-action. Since $G$ is
  geometrically finite, $x$ is either a conical limit point or bounded
  parabolic point for the $G$-action. In the former case, by
  \Cref{lem:HNN_conical_limit_seq_bounded} we conclude that there is a
  conical limit sequence of the form $(hg_k)$ for $x$, where $h \in G$
  and $g_k \in G_0$. Then $(g_k)$ is a conical limit sequence for $x$
  in $G_0$ and we are done.
  
  In the latter case, let $P < G$ be the stabilizer of $x$, a
  parabolic subgroup of $G$. We claim that in fact $P$ is a subgroup
  of $G_0$. If $x \in \Lambda(J_1) \cup \Lambda(J_{-1})$, then $x$
  lies in either $\del B_1$ or $\del B_{-1}$, and then this follows
  from \Cref{lem:HNN_boundary_invariance}. And, if $x = gy$ for
  $y \in \Lambda(J_1) \cup \Lambda(J_{-1})$ and $g \in G_0$, then the
  stabilizer of $x$ lies in $gG_0g^{-1} = G_0$. Finally, if
  $x \in \Lambda(G_0) \setminus G_0(\Lambda(J_1) \cup
  \Lambda(J_{-1}))$, then part \ref{item:HNN_prop_iii} of
  \Cref{prop:HNN_basic_properties} says that $x \in A_0$, and
  \Cref{lem:interactive_triple_dynamics} implies that no element of
  $G$ with positive length can fix a point in $A_0$.

  Now, since $x$ is a bounded parabolic point for the $G$-action on
  $\Lambda(G)$, local compactness of $\Lambda(G) \setminus \{x\}$
  implies that there is some compact $K \subset \Lambda(G)$ so that
  $P(K) = \Lambda(G) \setminus \{x\}$. We let
  $K_0 = K \cap \Lambda(G_0)$, which is a compact in
  $\Lambda(G_0) \setminus \{x\}$.
    
  Using $G_0$-invariance (and hence $P$-invariance) of $\Lambda(G_0)$,
  we now have that
  \[
    P(K_0) = P(K \cap \Lambda(G_0)) = P(K) \cap \Lambda(G_0) =
    \Lambda(G_0) \setminus \{x\}
  \]
  as desired.
\end{proof}
 
\bibliographystyle{alpha}
\bibliography{references}

\end{document}

%% file: horoball.pdf_tex
\begingroup%
  \makeatletter%
  \providecommand\color[2][]{%
    \errmessage{(Inkscape) Color is used for the text in Inkscape, but the package 'color.sty' is not loaded}%
    \renewcommand\color[2][]{}%
  }%
  \providecommand\transparent[1]{%
    \errmessage{(Inkscape) Transparency is used (non-zero) for the text in Inkscape, but the package 'transparent.sty' is not loaded}%
    \renewcommand\transparent[1]{}%
  }%
  \providecommand\rotatebox[2]{#2}%
  \newcommand*\fsize{\dimexpr\f@size pt\relax}%
  \newcommand*\lineheight[1]{\fontsize{\fsize}{#1\fsize}\selectfont}%
  \ifx\svgwidth\undefined%
    \setlength{\unitlength}{160.41712921bp}%
    \ifx\svgscale\undefined%
      \relax%
    \else%
      \setlength{\unitlength}{\unitlength * \real{\svgscale}}%
    \fi%
  \else%
    \setlength{\unitlength}{\svgwidth}%
  \fi%
  \global\let\svgwidth\undefined%
  \global\let\svgscale\undefined%
  \makeatother%
  \begin{picture}(1,1.06873445)%
    \lineheight{1}%
    \setlength\tabcolsep{0pt}%
    \put(0,0){\includegraphics[width=\unitlength,page=1]{horoball.pdf}}%
    \put(0.55204828,0.78797583){\color[rgb]{0,0,0}\makebox(0,0)[lt]{\lineheight{1.25}\smash{\begin{tabular}[t]{l}$B$\end{tabular}}}}%
    \put(0,0){\includegraphics[width=\unitlength,page=2]{horoball.pdf}}%
    \put(0.22489698,0.78984949){\color[rgb]{0,0,0}\makebox(0,0)[lt]{\lineheight{1.25}\smash{\begin{tabular}[t]{l}$w_k$\end{tabular}}}}%
    \put(0.17729934,0.85014629){\color[rgb]{0,0,0}\makebox(0,0)[lt]{\lineheight{1.25}\smash{\begin{tabular}[t]{l}$g_kx$\end{tabular}}}}%
    \put(0.4718176,0.30715716){\color[rgb]{0,0,0}\makebox(0,0)[lt]{\lineheight{1.25}\smash{\begin{tabular}[t]{l}$x$\end{tabular}}}}%
    \put(0.49242326,1.02130009){\color[rgb]{0,0,0}\makebox(0,0)[lt]{\lineheight{1.25}\smash{\begin{tabular}[t]{l}$z$\end{tabular}}}}%
    \put(0.36370355,0.50835421){\color[rgb]{0,0,0}\makebox(0,0)[lt]{\lineheight{1.25}\smash{\begin{tabular}[t]{l}$c_k$\end{tabular}}}}%
  \end{picture}%
\endgroup%

%% file: AFP_ex.pdf_tex
\begingroup%
  \makeatletter%
  \providecommand\color[2][]{%
    \errmessage{(Inkscape) Color is used for the text in Inkscape, but the package 'color.sty' is not loaded}%
    \renewcommand\color[2][]{}%
  }%
  \providecommand\transparent[1]{%
    \errmessage{(Inkscape) Transparency is used (non-zero) for the text in Inkscape, but the package 'transparent.sty' is not loaded}%
    \renewcommand\transparent[1]{}%
  }%
  \providecommand\rotatebox[2]{#2}%
  \newcommand*\fsize{\dimexpr\f@size pt\relax}%
  \newcommand*\lineheight[1]{\fontsize{\fsize}{#1\fsize}\selectfont}%
  \ifx\svgwidth\undefined%
    \setlength{\unitlength}{566.92913386bp}%
    \ifx\svgscale\undefined%
      \relax%
    \else%
      \setlength{\unitlength}{\unitlength * \real{\svgscale}}%
    \fi%
  \else%
    \setlength{\unitlength}{\svgwidth}%
  \fi%
  \global\let\svgwidth\undefined%
  \global\let\svgscale\undefined%
  \makeatother%
  \begin{picture}(1,0.5)%
    \lineheight{1}%
    \setlength\tabcolsep{0pt}%
    \put(0,0){\includegraphics[width=\unitlength,page=1]{AFP_ex.pdf}}%
    \put(0.51039349,0.42239892){\color[rgb]{0,0,0}\makebox(0,0)[lt]{\lineheight{1.25}\smash{\begin{tabular}[t]{l}$i\R$\end{tabular}}}}%
    \put(0,0){\includegraphics[width=\unitlength,page=2]{AFP_ex.pdf}}%
    \put(0.48666965,0.4656338){\color[rgb]{0,0,0}\makebox(0,0)[lt]{\lineheight{1.25}\smash{\begin{tabular}[t]{l}$\infty$\end{tabular}}}}%
    \put(0,0){\includegraphics[width=\unitlength,page=3]{AFP_ex.pdf}}%
    \put(0.22432327,0.39529099){\color[rgb]{1,0,0}\makebox(0,0)[lt]{\lineheight{1.25}\smash{\begin{tabular}[t]{l}$B_1$\end{tabular}}}}%
    \put(0.72432327,0.39529099){\color[rgb]{0,0,1}\makebox(0,0)[lt]{\lineheight{1.25}\smash{\begin{tabular}[t]{l}$B_2$\end{tabular}}}}%
    \put(0.80173012,0.26231895){\color[rgb]{0,0,0}\makebox(0,0)[lt]{\lineheight{1.25}\smash{\begin{tabular}[t]{l}$\R$\end{tabular}}}}%
    \put(0.21373556,0.459204){\color[rgb]{0,0,0}\makebox(0,0)[lt]{\lineheight{1.25}\smash{\begin{tabular}[t]{l}$\Lambda(J) = \{0, \infty\}$\end{tabular}}}}%
    \put(0,0){\includegraphics[width=\unitlength,page=4]{AFP_ex.pdf}}%
    \put(0.34954853,0.38625898){\color[rgb]{1,0,0}\makebox(0,0)[lt]{\lineheight{1.25}\smash{\begin{tabular}[t]{l}$g \in G_1 \setminus J$\end{tabular}}}}%
    \put(0.565788,0.22703369){\color[rgb]{1,0,0}\makebox(0,0)[lt]{\lineheight{1.25}\smash{\begin{tabular}[t]{l}$\Lambda(G_1)$\end{tabular}}}}%
    \put(0.36257351,0.22703369){\color[rgb]{0,0,1}\makebox(0,0)[lt]{\lineheight{1.25}\smash{\begin{tabular}[t]{l}$\Lambda(G_2)$\end{tabular}}}}%
    \put(0,0){\includegraphics[width=\unitlength,page=5]{AFP_ex.pdf}}%
    \put(0.53817596,0.1275148){\color[rgb]{0,0,1}\makebox(0,0)[lt]{\lineheight{1.25}\smash{\begin{tabular}[t]{l}$g \in G_2 \setminus J$\end{tabular}}}}%
    \put(0,0){\includegraphics[width=\unitlength,page=6]{AFP_ex.pdf}}%
  \end{picture}%
\endgroup%

%% file: AFP_T.pdf_tex
\begingroup%
  \makeatletter%
  \providecommand\color[2][]{%
    \errmessage{(Inkscape) Color is used for the text in Inkscape, but the package 'color.sty' is not loaded}%
    \renewcommand\color[2][]{}%
  }%
  \providecommand\transparent[1]{%
    \errmessage{(Inkscape) Transparency is used (non-zero) for the text in Inkscape, but the package 'transparent.sty' is not loaded}%
    \renewcommand\transparent[1]{}%
  }%
  \providecommand\rotatebox[2]{#2}%
  \newcommand*\fsize{\dimexpr\f@size pt\relax}%
  \newcommand*\lineheight[1]{\fontsize{\fsize}{#1\fsize}\selectfont}%
  \ifx\svgwidth\undefined%
    \setlength{\unitlength}{566.92913386bp}%
    \ifx\svgscale\undefined%
      \relax%
    \else%
      \setlength{\unitlength}{\unitlength * \real{\svgscale}}%
    \fi%
  \else%
    \setlength{\unitlength}{\svgwidth}%
  \fi%
  \global\let\svgwidth\undefined%
  \global\let\svgscale\undefined%
  \makeatother%
  \begin{picture}(1,0.5)%
    \lineheight{1}%
    \setlength\tabcolsep{0pt}%
    \put(0,0){\includegraphics[width=\unitlength,page=1]{AFP_T.pdf}}%
    \put(0.14945998,0.44712224){\color[rgb]{1,0,0}\makebox(0,0)[lt]{\lineheight{1.25}\smash{\begin{tabular}[t]{l}$B_1$\end{tabular}}}}%
    \put(0.81183327,0.44774935){\color[rgb]{0,0,1}\makebox(0,0)[lt]{\lineheight{1.25}\smash{\begin{tabular}[t]{l}$B_2$\end{tabular}}}}%
    \put(0,0){\includegraphics[width=\unitlength,page=2]{AFP_T.pdf}}%
    \put(0.47008148,0.40657345){\color[rgb]{0,0,0}\makebox(0,0)[lt]{\lineheight{1.25}\smash{\begin{tabular}[t]{l}$\Lambda(J)$\end{tabular}}}}%
    \put(0,0){\includegraphics[width=\unitlength,page=3]{AFP_T.pdf}}%
    \put(0.25719213,0.40288869){\color[rgb]{0,0,1}\makebox(0,0)[lt]{\lineheight{1.25}\smash{\begin{tabular}[t]{l}$T_{1,2}$\end{tabular}}}}%
    \put(0.68599365,0.40288778){\color[rgb]{1,0,0}\makebox(0,0)[lt]{\lineheight{1.25}\smash{\begin{tabular}[t]{l}$T_{1,1}$\end{tabular}}}}%
  \end{picture}%
\endgroup%

%% file: afp_parabolics_1.pdf_tex
\begingroup%
  \makeatletter%
  \providecommand\color[2][]{%
    \errmessage{(Inkscape) Color is used for the text in Inkscape, but the package 'color.sty' is not loaded}%
    \renewcommand\color[2][]{}%
  }%
  \providecommand\transparent[1]{%
    \errmessage{(Inkscape) Transparency is used (non-zero) for the text in Inkscape, but the package 'transparent.sty' is not loaded}%
    \renewcommand\transparent[1]{}%
  }%
  \providecommand\rotatebox[2]{#2}%
  \newcommand*\fsize{\dimexpr\f@size pt\relax}%
  \newcommand*\lineheight[1]{\fontsize{\fsize}{#1\fsize}\selectfont}%
  \ifx\svgwidth\undefined%
    \setlength{\unitlength}{350.62456822bp}%
    \ifx\svgscale\undefined%
      \relax%
    \else%
      \setlength{\unitlength}{\unitlength * \real{\svgscale}}%
    \fi%
  \else%
    \setlength{\unitlength}{\svgwidth}%
  \fi%
  \global\let\svgwidth\undefined%
  \global\let\svgscale\undefined%
  \makeatother%
  \begin{picture}(1,0.53197532)%
    \lineheight{1}%
    \setlength\tabcolsep{0pt}%
    \put(0,0){\includegraphics[width=\unitlength,page=1]{afp_parabolics_1.pdf}}%
    \put(0.68181877,0.3566848){\color[rgb]{0,0,0}\makebox(0,0)[lt]{\lineheight{1.25}\smash{\begin{tabular}[t]{l}$p$\end{tabular}}}}%
    \put(0.54401744,0.3137244){\color[rgb]{0,0,0}\makebox(0,0)[lt]{\lineheight{1.25}\smash{\begin{tabular}[t]{l}$\Lambda(G_1)$\end{tabular}}}}%
    \put(0.89875235,0.16436968){\color[rgb]{0,0,0}\makebox(0,0)[lt]{\lineheight{1.25}\smash{\begin{tabular}[t]{l}$K_3$\end{tabular}}}}%
    \put(0.73543487,0.33196719){\color[rgb]{0,0,0}\makebox(0,0)[lt]{\lineheight{1.25}\smash{\begin{tabular}[t]{l}$K_1$\end{tabular}}}}%
    \put(0.23405417,0.00553759){\color[rgb]{0,0,0}\makebox(0,0)[lt]{\lineheight{1.25}\smash{\begin{tabular}[t]{l}$K_2$\end{tabular}}}}%
  \end{picture}%
\endgroup%

%% file: HNN_T.pdf_tex
\begingroup%
  \makeatletter%
  \providecommand\color[2][]{%
    \errmessage{(Inkscape) Color is used for the text in Inkscape, but the package 'color.sty' is not loaded}%
    \renewcommand\color[2][]{}%
  }%
  \providecommand\transparent[1]{%
    \errmessage{(Inkscape) Transparency is used (non-zero) for the text in Inkscape, but the package 'transparent.sty' is not loaded}%
    \renewcommand\transparent[1]{}%
  }%
  \providecommand\rotatebox[2]{#2}%
  \newcommand*\fsize{\dimexpr\f@size pt\relax}%
  \newcommand*\lineheight[1]{\fontsize{\fsize}{#1\fsize}\selectfont}%
  \ifx\svgwidth\undefined%
    \setlength{\unitlength}{566.92913386bp}%
    \ifx\svgscale\undefined%
      \relax%
    \else%
      \setlength{\unitlength}{\unitlength * \real{\svgscale}}%
    \fi%
  \else%
    \setlength{\unitlength}{\svgwidth}%
  \fi%
  \global\let\svgwidth\undefined%
  \global\let\svgscale\undefined%
  \makeatother%
  \begin{picture}(1,0.5)%
    \lineheight{1}%
    \setlength\tabcolsep{0pt}%
    \put(0,0){\includegraphics[width=\unitlength,page=1]{HNN_T.pdf}}%
    \put(0.19008497,0.36642756){\color[rgb]{1,0,0}\makebox(0,0)[lt]{\lineheight{1.25}\smash{\begin{tabular}[t]{l}$B_1$\end{tabular}}}}%
    \put(0.77546082,0.36642756){\color[rgb]{0,0.23529412,1}\makebox(0,0)[lt]{\lineheight{1.25}\smash{\begin{tabular}[t]{l}$B_{-1}$\end{tabular}}}}%
    \put(0,0){\includegraphics[width=\unitlength,page=2]{HNN_T.pdf}}%
    \put(0.48954853,0.46203369){\color[rgb]{0,0,0}\makebox(0,0)[lt]{\lineheight{1.25}\smash{\begin{tabular}[t]{l}$A$\end{tabular}}}}%
    \put(0,0){\includegraphics[width=\unitlength,page=3]{HNN_T.pdf}}%
  \end{picture}%
\endgroup%

%% file: HNN_parabolics_2.pdf_tex
\begingroup%
  \makeatletter%
  \providecommand\color[2][]{%
    \errmessage{(Inkscape) Color is used for the text in Inkscape, but the package 'color.sty' is not loaded}%
    \renewcommand\color[2][]{}%
  }%
  \providecommand\transparent[1]{%
    \errmessage{(Inkscape) Transparency is used (non-zero) for the text in Inkscape, but the package 'transparent.sty' is not loaded}%
    \renewcommand\transparent[1]{}%
  }%
  \providecommand\rotatebox[2]{#2}%
  \newcommand*\fsize{\dimexpr\f@size pt\relax}%
  \newcommand*\lineheight[1]{\fontsize{\fsize}{#1\fsize}\selectfont}%
  \ifx\svgwidth\undefined%
    \setlength{\unitlength}{334.70217216bp}%
    \ifx\svgscale\undefined%
      \relax%
    \else%
      \setlength{\unitlength}{\unitlength * \real{\svgscale}}%
    \fi%
  \else%
    \setlength{\unitlength}{\svgwidth}%
  \fi%
  \global\let\svgwidth\undefined%
  \global\let\svgscale\undefined%
  \makeatother%
  \begin{picture}(1,0.63680399)%
    \lineheight{1}%
    \setlength\tabcolsep{0pt}%
    \put(0,0){\includegraphics[width=\unitlength,page=1]{HNN_parabolics_2.pdf}}%
    \put(0.08774152,0.35295577){\color[rgb]{0,0,0}\makebox(0,0)[lt]{\lineheight{1.25}\smash{\begin{tabular}[t]{l}$f(B_1)$\end{tabular}}}}%
    \put(0.74108937,0.07733058){\color[rgb]{0,0,0}\makebox(0,0)[lt]{\lineheight{1.25}\smash{\begin{tabular}[t]{l}$B_{-1}$\end{tabular}}}}%
    \put(0.25314045,0.0763138){\color[rgb]{0,0,0}\makebox(0,0)[lt]{\lineheight{1.25}\smash{\begin{tabular}[t]{l}$B_1$\end{tabular}}}}%
    \put(0.4633688,0.38565112){\color[rgb]{0,0,0}\makebox(0,0)[lt]{\lineheight{1.25}\smash{\begin{tabular}[t]{l}$p$\end{tabular}}}}%
    \put(0,0){\includegraphics[width=\unitlength,page=2]{HNN_parabolics_2.pdf}}%
    \put(0.147636,0.55367712){\color[rgb]{0,0.02745098,1}\makebox(0,0)[lt]{\lineheight{1.25}\smash{\begin{tabular}[t]{l}$K_2$\end{tabular}}}}%
    \put(0,0){\includegraphics[width=\unitlength,page=3]{HNN_parabolics_2.pdf}}%
    \put(0.16008846,0.19591391){\color[rgb]{1,0,0}\makebox(0,0)[lt]{\lineheight{1.25}\smash{\begin{tabular}[t]{l}$K_3$\end{tabular}}}}%
    \put(0,0){\includegraphics[width=\unitlength,page=4]{HNN_parabolics_2.pdf}}%
  \end{picture}%
\endgroup%